\documentclass[a4paper,11pt,reqno]{amsart}

\usepackage{
	amsfonts,
	amsmath,
	amsopn,
	amssymb,
	amsthm,
	bbm,
	dsfont,
	enumitem,
	graphicx,
	mathrsfs,
	mathtools,
	soul,
	verbatim,
	xcolor,
	xspace,
	latexsym,
 subcaption,
}
\usepackage{tikz}
\usetikzlibrary{decorations.pathreplacing}
\usetikzlibrary{shapes.geometric,calc,decorations.markings,math}
\tikzstyle{nodo}=[circle,draw,fill,inner sep=0pt,minimum size=%
1.5mm]
\tikzstyle{infinito}=[circle,inner sep=0pt,minimum size=0mm]
\usepackage{xcolor}
\usepackage{geometry}
\usepackage[utf8]{inputenc}

\mathtoolsset{showonlyrefs}

\newtheorem{theorem}{Theorem}[section]
\newtheorem{lemma}[theorem]{Lemma}
\newtheorem{proposition}[theorem]{Proposition}

\newtheorem*{problem}{Problem}

\theoremstyle{remark}
\newtheorem{remark}[theorem]{Remark}
\newtheorem*{remark*}{Remark}

\theoremstyle{definition}
\newtheorem{definition}[theorem]{Definition}

\DeclarePairedDelimiter\abs{\lvert}{\rvert}%

\newcommand{\R}{\mathbb{R}}

\newcommand{\Rd}{\mathbb{R}^2}
\newcommand{\Rp}{\mathbb{R}^+}
\newcommand{\C}{\mathbb{C}}

\newcommand{\N}{\mathbb{N}}
\newcommand{\NN}{\mathcal{N}}

\newcommand{\D}{\mathcal{D}}
\newcommand{\Eps}{\mathcal{E}}
\newcommand{\G}{\mathcal{G}}

\newcommand{\I}{\mathcal{I}}
\newcommand{\F}{\mathcal{F}}

\renewcommand{\Re}{\mathrm{Re}}
\renewcommand{\leq}{\leqslant}
\renewcommand{\geq}{\geqslant}

\newcommand{\la}{\lambda}
\newcommand{\La}{\Lambda}

\newcommand{\al}{\alpha}

\newcommand{\om}{\omega}

\newcommand{\lap}{\Delta}
\newcommand{\na}{\nabla}

\newcommand{\f}[2]{\frac{#1}{#2}}

\newcommand{\scal}[2]{\langle #1,#2\rangle}

\newcommand{\deb}{\rightharpoonup}

\newcommand{\norma}[1]{{\left\Vert#1\right\Vert}}

\usepackage[hidelinks]{hyperref}

\title{Normalized NLS ground states on a double plane hybrid}

\author[F. Boni]{Filippo Boni}
\address{Scuola Superiore Meridionale, Largo S. Marcellino, 10, 80138, Napoli, Italy}

\email{f.boni@ssmeridionale.it}

\author[R. Carlone]{Raffaele Carlone}
\address{Università degli Studi di Napoli Federico II, Dipartimento di Matematica e Applicazioni ``Renato Caccioppoli”, Via Cintia, Monte S. Angelo, 80126, Napoli, Italy}
\email{raffaele.carlone@unina.it}

\author[I. Di Giorgio]{Ilenia Di Giorgio}
\address{Università degli Studi di Napoli Federico II, Dipartimento di Matematica e Applicazioni ``Renato Caccioppoli”, Via Cintia, Monte S. Angelo, 80126, Napoli, Italy}
\email{ilenia.digiorgio@unina.it}

\date{\today}

\begin{document}
	
	
	\begin{abstract}
We investigate the existence and the properties of normalized ground states of a nonlinear Schr\"odinger equation on a quantum hybrid formed by two planes connected at a point. The nonlinearities are of power type and $L^2$-subcritical, while the matching condition between the two planes generates two point interactions of different strengths on each plane, together with a coupling condition between the two planes. We prove that ground states exist for every value of the mass and two different qualitative situations are possible depending on the matching condition: either ground states concentrate on one of the plane only, or ground states distribute on both the planes and are positive, radially symmetric, decreasing and present a logarithmic singularity at the origin of each plane. Moreover, we discuss how the mass distributes on the two planes and compare the strengths of the logarithmic singularities on the two planes when the parameters of the matching condition and the powers of the nonlinear terms vary.  
	\end{abstract}
	
	\maketitle	
	
	
	
	\section{Introduction}


Quantum hybrids are multidimensional structures on which a quantum dynamics is set: more specifically, different manifolds characterized by different dimensionality are glued together in such a way that a non-trivial dynamics on the whole structure can be defined. The significance of this process of hybridization lies in the specific mathematical technique used to construct such structures that allows us to deal with operators typical of quantum mechanics.

In the present paper, we continue the investigation of stationary solutions of nonlinear Schr\"odinger equations on quantum hybrids started in \cite{ABCT-24,ABCT-24-2}, where a halfline-plane hybrid and a line-plane hybryd were considered (see Figure \ref{fig:ibridiprec}). Here we consider a double plane hybrid, namely a hybrid structure where two identical planes are connected at a single point (see Figure \ref{fig:ibrido-doppiopiano}). The dynamics is described by a nonlinear 
Schr\"odinger equation of the form
\begin{equation}
		\label{1an}
		i\f{\partial}{\partial t}\Psi = H\Psi + g(\Psi), \qquad \,\,\Psi =(z,w):L^2(\Rd)\oplus L^2(\Rd)\to\C^2
	\end{equation}
	where $g:\C^2\to \C^2$ is a nonlinear function given by
\[
g(z,w) = (-\abs{z}^{p_1-2}z,-\abs{w}^{p_2-2}w), \qquad p_1,p_2>2
\]
	and $H$ is a self-adjoint operator on the quantum hybrid (see Appendix \ref{app:operator} for more details). Such a dynamics coincides with the standard NLS dynamics on each plane far from the points where the planes are glued together and allow a connection between the two planes through proper matching conditions that make the operator $H$ a self-adjoint operator.

\begin{figure}[ht]
    \centering
        \includegraphics[width=0.7\textwidth]{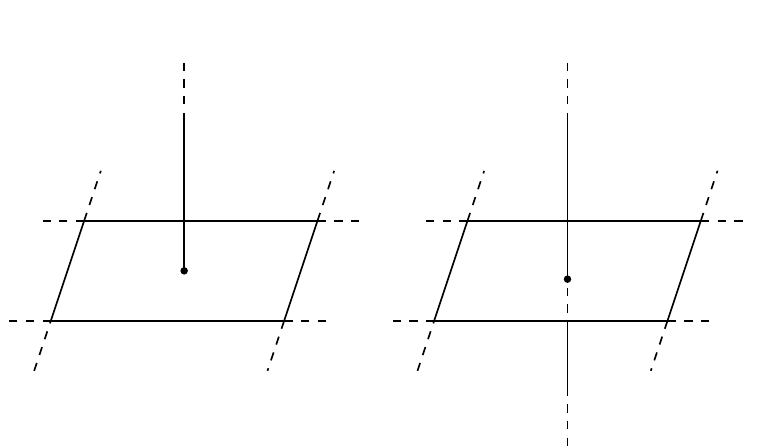}
        \caption{Halfline-plane hybrid (on the left) and line-plane hybrid (on the right).}
        \label{fig:ibridiprec}
        \vspace{1cm}
        \includegraphics[width=0.7\textwidth]{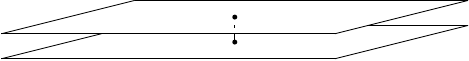}
        \caption{Double-plane hybrid.}
    \label{fig:ibrido-doppiopiano}
\end{figure}

From the physical point of view, the study of various models involving quantum hybrids began in the context of spectroscopy in order to understand the characteristics of the conduction in regions where the geometry presents a singularity of halfline-plane type (see \cite{JvGW-80} and Figure \ref{fig:ibridiprec}).

The mathematical interest in this type of models dates back to the late eighties with the papers \cite{ES-88,ES-87,ES-86}, where the authors considered both hybrids made of a plane and a halfline and hybrids made of two planes connected at a point. The geometry of point contacts is crucial in point-contact spectroscopy since it directly influences the mechanisms of electron transport. When the dimensions of the metallic contact are comparable to or smaller than the electron mean free path, the electron transport becomes ballistic rather than diffusive. In this regime, electrons can traverse the contact without scattering, allowing for a direct study of electron-phonon interactions \cite{JvGW-80}.

Although initially developed for its applications in the latest advancements of solid-state physics, the  technique of point-contact spectroscopy has proven to be useful for the characterization of superconducting materials, and in particular for mapping bosonic and superconducting order parameter spectra through quasiparticle classical and Andreev scattering, respectively \cite{pics}.

Hybrid quantum systems represent a possible model for materials or nanostructures \cite{FDNST} with spatial structure with mixed dimensionality.
The electrical properties of nano contacts affect the controllability, reliability, and efficiency of the  device. Managing current transport in nano contacts is essential for achieving technological advancements \cite{BWZ}. This effectively gives rise to structures that are geometrically similar to hybrids, whose dimensions make quantum effects non-negligible.

From a mathematical point of view, the study of linear and nonlinear PDEs on hybrids represents the generalization and the natural continuation of the well established line of research about linear and nonlinear dynamics on metric graphs. The literature in this field is huge and we provide only some recent results about this topic, especially in the nonlinear setting \cite{ABD,acfn_aihp,ADST,ASTcpde,AST,BMP,BLS,BD22,BCT,CJS,DSTaim,DT22,MNS,PS20}.

\noindent The common and peculiar aspect of the research programs on hybrids and metric graphs lie in the fact that, besides being inspired by physical models, they serve as useful laboratories for the development and application of mathematical techniques. The main difference between metric graphs and hybrids is that in the latter case the definition of self-adjoint operators is more difficult due to the singular character of the fundamental solution of the Laplacian in dimension higher than one, making the nonlinear problem even more challenging. Even though the present paper discusses a specific example of quantum hybrid as done in \cite{ABCT-24,ABCT-24-2}, the long-term goal is to isolate topological and dimensional properties that distinguish different quantum hybrids from the point of view of the existence and stability of stationary solutions.


	\subsection{Setting and main results.}

\noindent 
 The main goal of the paper is to study the existence and the properties of particular solutions of \eqref{1an}, namely ground states of the energy associated with \eqref{1an} subject to the constraint of fixed $L^2$ norm. The precise definition of the energy and of ground states is not straightforward and requires to be introduced step by step.

Let us start recalling the definition of the standard NLS energy on the plane, that is
\begin{equation}
			\label{uf}
			E_p(u):=\frac{1}{2} \int_{\R^2}\abs{\nabla u}^2 \,dx -\frac{1}{p}\int_{\R^2}\abs{u}^{p} \, dx,\,\qquad u\in H^1(\R^2),
\end{equation}
and suppose that the energy on the hybrid presents an interaction between the two planes that is concentrated only at the junction point: this interaction produces a point interaction on each plane together with a coupling term between the two planes. 

The point interaction on each plane can be defined through the theory of self-adjoint extensions of hermitian operators introduced by Krein and von Neumann (see \cite{RSII-80}) and produces a peculiar decomposition of the energy domain $\D$, namely every function $u\in \D$ decomposes as the sum of a \emph{regular part} $\phi\in H^1(\Rd)$ and a \emph{singular part} $q\G_1$, where $q\in\C$ is usually called the \emph{charge} and $\G_1$ is the Green's function of $-\lap+1$: more specifically, $\D$ reads as
\[
\D:= \{u\in L^2(\R^2):\exists\, q\in\C\,:\, u-q\G_1=:\phi\in H^1(\R^2)\}.
\]

The NLS energy with point interaction on the plane is thus defined as the functional $F_{p,\sigma}:\D\to \R$ acting as
\begin{equation}
\label{eq:F-psigma}
\begin{split}
F_{p,\sigma}(u):=&\frac{1}{2}Q_\sigma(u)-\frac{1}{p}\|u\|_{L^p(\Rd)}^p\\
=&\frac{1}{2}\left(\|\phi\|_{H^1(\Rd)}^2-\|u\|_{L^2(\Rd)}^2\right)+\frac{|q|^2}{2}\left(\sigma+\frac{\gamma-\log(2)}{2\pi}\right)-\frac{1}{p}\|u\|_{L^p(\Rd)}^p,
\end{split}
\end{equation}
with $\gamma$ being the Euler-Mascheroni constant. The parameter $\sigma$, later referred as the parameter of the point interaction, is obtained by applying the technique of self-adjoint extension of symmetric operators to the Laplace operator $-\Delta$ on $C^\infty_c(\Rd\setminus\{0\})$ (see the monograph \cite{AGHKH-88} for more details). Let us observe that the value of the parameter $\sigma$ for which the standard NLS energy is recovered is formally given by $\sigma=+\infty$: indeed, this choice forces the charge $q$ to be zero, so that the domain $\D$ reduces to $H^1(\Rd)$.  The existence and some stability properties of the ground states of the energy \eqref{eq:F-psigma} have been studied in \cite{ABCT,FGI-22} and, in presence of a nonlocal nonlinearity, in \cite{GMS}. Moreover, well-posedness, scattering properties and blow-up of solutions of the associated nonlinear Schr\"odinger equation have been investigated in \cite{CFN, CFN-22,FN-22}.   

Given $p_i>2$, $\sigma_i\in \R$ for $i=1,2$ and $\beta\geq 0$, and denoted by $U=(u_1,u_2)\in \La:=\D\times\D$ and by $q_i:=\lim_{|x|\to 0}\frac{u_i(x)}{\G_1(x)}$ the charges associated with $u_i$ for $i=1,2$, one can define the NLS energy on the double-plane hybrid as the functional $F: \Lambda\to \R$ acting as
\begin{equation}
\label{eq:F}
F(U):=\sum_{i=1,2}F_{p_i,\sigma_i}(u_i)-\beta \Re\{q_1\overline{q_2}\}.
 \end{equation}

 As one can notice, the coupling term is chosen to be quadratic as function of the charges $q_i$: by this choice, the operator $H$ associated with the quadratic part of the energy turns out to be self-adjoint (see Appendix \ref{app:operator}). Let us observe that, denoted with
 \begin{equation}
 \label{eq:quad-form}
     Q(U):=\sum_{i=1,2}Q_{\sigma_i}(u_i)-2\beta \Re\{q_1\overline{q_2}\},
 \end{equation}
 then the energy \eqref{eq:F} can be rewritten also as
 \begin{equation}
     F(U)=\frac{1}{2}Q(U)-\sum_{i=1,2}\frac{1}{p_i}\|u_i\|_{L^{p_i}(\Rd)}^{p_i}.
 \end{equation}
 Let us now give the definition of ground state. 

\begin{definition}(Ground state). 
		For every fixed $\mu>0$, we call \textit{ground state} of $F$ at mass $\mu$ any function $U$ belonging to
  \begin{equation*}
      \La_{\mu}:=\left\{(u_1,u_2)\in\La\,:\,\|u_1\|_{L^2(\R^2)}^2+\|u_2\|_{L^2(\R^2)}^2=\mu\right\}
  \end{equation*} 
  such that
		\[
		F(U)=\F(\mu):= \inf_{U\in\La_{\mu}} F(U)
		\]
		In particular, $\F(\mu)$ is called \textit{ground state energy level} of $F$ at mass $\mu$. 
	\end{definition}

Let us highlight that ground states $U$ of \eqref{eq:F} at mass $\mu$ solve for some $\omega\in \R$ the system
\begin{equation}
	\begin{cases}
		\label{eq:stat}
		H\,U + g(U) + \omega U =0\\
		U\in D(H).
\end{cases}\end{equation}
See Appendix \ref{app:operator} for a more explicit expression of \eqref{eq:stat}.
In particular, given a solution of \eqref{eq:stat}, one can construct \emph{standing waves} solutions of \eqref{1an}, i.e. solutions of the the form $\Psi(t,x,y)=e^{i\omega\, t}\,U(x,y)$  with $x,y\in\R^2$.

Summing up, we investigate the following problem:

\begin{problem}
    Given $p_i\in (2,4)$, $\sigma_i\in \R$ for $i=1,2$, $\beta\geq 0$ and $\mu>0$, study the existence and qualitative properties of ground states at mass $\mu$ of the energy \eqref{eq:F}, namely minimizers of
    \begin{equation}
    \begin{split}
        F(U)&=\sum_{i=1,2}F_{p_i,\sigma_i}(u_i)-\beta\Re\{q_1\overline{q_2}\}\\
        &=\sum_{i=1,2}\biggl\{\frac{1}{2}Q_{\sigma_i}(u_i)-\frac{1}{p_i}\|u_i\|_{L^{p_i}}^{p_i}\biggl\}-\beta \Re\{q_1\overline{q_2}\}
    \end{split}
    \end{equation}
    among all the functions belonging to $\Lambda_\mu$.
\end{problem}

Let us highlight that the limitations on the values of the powers $p_i$ for $i=1,2$ reduce the study to the $L^2$-subcritical case, namely the case where the energy $F$ is bounded from below for every value of the mass $\mu>0$: this can be proved in general by applying the so-called Gagliardo-Nirenberg type inequalities (see Subsection \ref{subsec:GN}).

In order to ease the notation, in the following we will often use $F_i$ in place of $F_{p_i,\sigma_i}$. Moreover, we will denote the ground state energy levels of $E_p$ and $F_i$ at mass $\mu$ respectively as
\begin{equation}
    \Eps_p(\mu):=\inf_{v\in H^1_\mu(\Rd)} E_p(v)
\end{equation}
and
\begin{equation}
\F_i(\mu):=\inf_{v\in\D_\mu} F_i(v),\quad i=1,2,
\end{equation}
where the subscript $\mu$ in $H^1(\Rd)$ and $\D$ denotes that the additional constraint $\|v\|_{L^2(\Rd)}^2=\mu$ holds. 
\noindent
We present now the main results of the paper. We recall that the parameters in \eqref{eq:F} are chosen as follows when not specified: $p_i\in (2,4)$, $\sigma_i\in \R$ for $i=1,2$, and $\beta\geq 0$.

\noindent
The first result concerns the existence of ground states at fixed mass.

	\begin{theorem}[Existence of ground states]
        \label{thm:existence}
		For every $\mu>0$, there exists at least a ground state of \eqref{eq:F} at mass $\mu$.
	\end{theorem}

Differently from the case of other hybrids considered \cite{ABCT-24,ABCT-24-2}, ground states exist for every value of the mass when two planes are connected at a single point. This is due to the fact that point interactions on the plane are always attractive, namely the associated self-adjoint operators on the plane have always a negative eigenvalue, and this makes more convenient to concentrate around the point of contact rather than escaping at infinity.

 In the second result, we provide a characterization of ground states at fixed mass both when $\beta=0$ and $\beta>0$. These two cases correspond to the decoupled and coupled case respectively and are coherent with what happens in the linear setting \cite{ES-86}, in the sense that the actual hybridization of the two planes happen when the coupling parameter $\beta$ is different from zero.
 
 \begin{theorem}
		\label{thm:charact}
		Let $\mu>0$. If $\beta=0$, then
		\begin{equation}
			\label{eq:Fmu=F12mu}
			\F(\mu)=\min \left\{\F_{1}(\mu), \F_{2}(\mu)\right\}.
		\end{equation}
	Moreover, every ground state $U$ is of the form 
    \begin{enumerate}
\item[i)] $U=(u_1,0)$, \quad if $\F_{{1}}(\mu)<\F_{2}(\mu)$;
\item[ii)] $U=(0,u_2)$,\quad if $\F_{{1}}(\mu)>\F_{2}(\mu)$;
\item[iii)] either $U=(u_1,0)$ or $U=(0,u_2)$, if $\F_{{1}}(\mu)=\F_{2}(\mu)$,
	\end{enumerate}
    and the non-zero component of $U$ is positive, radially symmetric and decreasing along the radial direction and has positive charge, up to a multiplication by a phase factor.
    
If instead $\beta>0$, then 
\begin{equation}
\label{eq:Fmu<F12mu}
    \F(\mu)<\min\left\{\F_{1}(\mu)\,,\,\F_{2}(\mu)\right\}.
\end{equation}
In particular, given $U=(u_1,u_2)$ a ground state of $F$ at mass $\mu$ and $q_i$ be the charge associated with $u_i$, then, up to a multiplication by a phase factor, $q_i>0$ and $u_i$ is positive, radially symmetric and decreasing along the radial direction for $i=1,2$. Furthermore, denoted by $\omega$ the Lagrange multiplier associated with $U$, by $\G_\omega$ the Green's function of $-\lap+\omega$ and by $\phi_{i,\omega}:=u_i-q_i\G_\omega$ for $i=1,2$, then also $\phi_{i,\omega}$ is positive, radially symmetric and decreasing along the radial direction. 
	\end{theorem}

 \begin{remark}
      As explained in Subsection \ref{subsec:domain}, the decomposition consider in Theorem \ref{thm:charact}, i.e   $u_i=\phi_{i,\omega}+q_i\G_\omega$, is another possible decomposition for $u_i$, that differs from the one presented in the Introduction from the fact that the singular term is given by the Green's function of $-\lap+\omega$ instead of the Green's function of $-\lap+1$. 
\end{remark}

In the next result, we investigate the distribution of the charges and of the mass on the hybrid when the parameter of the point interaction on the first plane is lower than the same parameter on the second plane and  the power nonlinearities are the same on each plane. 

	\begin{theorem}
		\label{thm:p1=p2}
	Let $\sigma_1<\sigma_2$, $\beta>0$ and $p_1=p_2$ and let $U=(u_1,u_2)$ be a ground state of \eqref{eq:F} at mass $\mu>0$. For $i=1,2$, denote by $q_i$ the charge associated with $u_i$ and by $\mu_i$ the mass of $u_i$. Therefore, it holds that
	\begin{enumerate}
\item[i)] $q_2<q_1$;
\item[ii)] for any $\sigma_1\in \R$ and $\sigma_2\to+\infty$, it holds that 
\begin{equation}
    \|u_1\|_{L^2(\Rd)}^2\to \mu\quad \text{and} \quad F(U)\to \F_1(\mu); 
\end{equation}

	\end{enumerate} 
	\end{theorem}

In the last theorem, we investigate the distribution of mass on the hybrid when the parameters of the point interactions on the two planes are equal and sufficiently large and the power nonlinearity on the first plane is smaller than the one on the second. Before stating the result, let us introduce for every $p_1\neq p_2$  the critical mass
 \begin{equation}
 \label{mu-star}
\mu^\star:=\left(\f{\Eps_{p_1}(1)}{\Eps_{p_2}(1)}\right)^{\f{(4-p_1)(4-p_2)}{2(p_2-p_1)}}.
		\end{equation}

	\begin{theorem}
	\label{thm:sigma1=sigma2}
	Let $\sigma_1=\sigma_2=\sigma$, $p_1<p_2$, $\mu>0$ and $U=(u_1,u_2)$ be a ground state of $F$ at mass $\mu>0$. Then, as $\sigma\to+\infty$, the following limits hold:
	\begin{enumerate}
		\item[i)] if $\mu< \mu^\star$,  then 
  \begin{equation}
      \|u_1\|_{L^2(\Rd)}^2\to \mu\quad \text{and}\quad F(U)\to\Eps_{p_1}(\mu),
  \end{equation} 
		\item[ii)]  if $\mu>\mu^\star$, then 
  \begin{equation}
     \|u_2\|_{L^2(\Rd)}^2\to \mu\quad \text{and}\quad F(U)\to\Eps_{p_2}(\mu).
  \end{equation}
	\end{enumerate}
\end{theorem}

Theorem \ref{thm:sigma1=sigma2} and Theorem \ref{thm:p1=p2} are the first results on the distribution of the mass and of the charges on a quantum hybrid. The results are obtained in some specific regimes of the parameters only, and this is due to technical reasons. One of the main obstacle is due to the presence of the point interaction, that breaks the invariance by dilations of the equation and does not allow to perform the standard scalings associated with the Nonlinear Schr\"odinger Equation. Therefore, it becomes difficult to obtain precise informations on the infimum of the NLS energy with point interaction as function of the mass and the other parameters, and this seems a crucial ingredient to understand the problem in more generality.  \\

The paper is organized as follows:
\begin{itemize}[label=$-$]
 \item in Section \ref{sec:prelim} we recall some preliminary results, in particular concerning the domain of the energy with point interaction in Subsection \ref{subsec:domain}, the problem at infinity on the hybrid in Subsection \ref{subsec:probinf} and Gagliardo-Nirenberg inequalities in Subsection \ref{subsec:GN};
 \item in Section \ref{sec:prfthm12} we provide the proof of Theorem \ref{thm:existence} about the existence of ground states.
 \item Section \ref{sec:prfthm13} contains the proof of Theorem \ref{thm:charact} abot the properties of ground states; 
 \item in Section \ref{sec:prfthm1516} we present the proofs of Theorem \ref{thm:p1=p2} and Theorem \ref{thm:sigma1=sigma2}, concerning the distribution of mass and the comparison between the strength of the logarithmic singularities on the two planes. 
\end{itemize}
The paper ends with Appendix \ref{app:operator} about the rigorous definition of the class of self-adjoint operators on the hybrid. 

\section{Preliminaries}
 \label{sec:prelim}
  
In this section, we recall some well-known facts and results that will be useful along the paper.\\

\subsection{About the representation of the domain $\D$}

\label{subsec:domain}

First of all, let us recall that every function $u\in \D$ can be written as the sum of a \emph{regular part} belonging to $H^1(\Rd)$ and a \emph{singular part} given by a complex multiple of $\G_1$, which is the Green's function of $-\lap+1$. This representation is only one possible way to represent functions in $\D$. Indeed, if one denotes by $\G_\la$ the Green's function of $-\lap+\la$ for any $\la>0$, then one can write $u$ as
\begin{equation}
\label{eq:repr-lambda}
  u=\phi_\la+q\G_\la,\quad \text{with}\quad \phi_\la:=\phi+q\left(\G_1-\G_\la\right)\in H^1(\Rd).  
\end{equation}
Moreover, it is not hard to see that the quadratic form $Q_\sigma$ can be written as function of $\phi_\la$ and $q$ as 
\begin{equation*}
    Q_\sigma(u)=\|\na \phi_\la\|_{L^2(\Rd)}^2+\la\|\phi_\la\|_{L^2(\Rd)}^2-\la\|u\|_{L^2(\Rd)}^2+|q|^2\left(\sigma+\theta_\la\right),
\end{equation*}
with 
\begin{equation*}
 \theta_\la:=\frac{\log({\sqrt{\la}/2})+\gamma}{2\pi}.
 \end{equation*}
Therefore, the domain and the action of the energy $F_{p,\sigma}$ can be written respectively as
\begin{equation*}
\D= \{u\in L^2(\R^2)\,:\,\exists\, q\in\C, \lambda>0 : u-q\mathcal{G}_{\lambda}:=\phi_{\lambda}\in H^1(\R^2)\}
\end{equation*}
and
\begin{equation*}
  F_{p,\sigma}(u)=\frac{1}{2}\left(\|\na \phi_\la\|_{L^2(\Rd)}^2+\la\|\phi_\la\|_{L^2(\Rd)}^2-\la\|u\|_{L^2(\Rd)}^2\right)+\frac{|q|^2}{2}\left(\sigma+\theta_\la\right)-\frac{1}{p}\|u\|_{L^p(\Rd)}^p  .
\end{equation*}

This alternative way of representing functions belonging to $\D$ is very important in some of the proofs of the paper (see for example Theorem \ref{thm:existence}).

	\subsection{The problem at infinity on the hybrid.}

 \label{subsec:probinf}
	Particularly relevant in such minimization problems is the so-called problem at infinity.
 
	Let us consider the energy $F^\infty:H^1(\Rd)\times H^1(\Rd)\to \R$, defined as
    \begin{equation}
		\label{eq:Finf}
		F^{\infty}(U):= E_{p_1}(u_1) \,+\, E_{p_2}(u_2),
	\end{equation}
 and the associated ground states energy level
	\begin{equation}
		\label{eq:inf-Finf}
		\mathcal{F}^{\infty} (\mu):= \inf_{\substack{U\in H^1(\Rd)\times H^1(\Rd) \\ \norma{U}^2_{2}=\mu}} F^{\infty}(U).
	\end{equation}

	In the next lemma, we provide a first characterization of \eqref{eq:inf-Finf} when $p_i\in(2,4)$.
	\begin{lemma}
		\label{lem:F-inf}
		Let $p_i\in (2,4)$ for $i=1,2$. Then for every $\mu>0$
		\begin{equation}
			\label{eq:Finf-Ep}
			\F^{\infty}(\mu)=\min\{\mathcal{E}_{p_1}(\mu),\mathcal{E}_{p_2}(\mu)\}.
		\end{equation}
	\end{lemma}
	
	\begin{proof}
		Let us first observe that the minimization problem \eqref{eq:inf-Finf} is decoupled and can be reduced to an algebraic problem, in particular
		\[
		\F^{\infty}(\mu)=\inf\{\mathcal{E}_{p_1}(\mu_1) + \mathcal{E}_{p_2}(\mu_2):\mu_1 + \mu_2 =\mu\}.
		\]
		It is well known that $\Eps_p(\mu)<0$ for every $\mu>0$ and $p>2$ and $\Eps_p(\mu)>-\infty$ for every $\mu>0$ and $p$ belonging to the $L^2$-subcritical regime, i.e. $2<p<4$. Moreover, by scaling properties of the energy, it follows that $\Eps_p(\mu)=-\rho_p \mu^{\frac{2}{4-p}}$, with $\rho_p=-\Eps_p(1)$, thus
		\[
		\F^{\infty} (\mu) = \inf_{m\in[0,\mu]} g(m),
		\]
		where 
		\[
		g(m)= -\rho_{p_1}\,m^{\frac{2}{4-p_1}} - \rho_{p_2}\,(\mu-m)^{\frac{2}{4-p_2}}.
		\]
		The function $g$ is continuous in $[0,\mu]$, of class $C^2$ in $(0,\mu)$ and 
		\begin{align}
			g''(m)&= -\frac{2(p_1-2)}{(4-p_1)^2}\,\rho_{p_1} m^{\frac{2(p_1-3)}{4-p_1}} - \frac{2(p_2-2)}{(4-p_2)^2} \,\rho_{p_2}(\mu-m)^{\frac{2(p_2-3)}{4-p_2}} <0 \quad \forall\, m\in (0,\mu), 
		\end{align}
		hence $g$ attains the minimum either at $0$ or at $\mu$, entailing the thesis.
		\end{proof}
 
  	In the next proposition, we give a more precise characterization of \eqref{eq:inf-Finf} when $p_1\neq p_2$.
   
		\begin{proposition}
			\label{livellidienergia}
			Let $p_i\in (2,4)$ for $i=1,2$, with $p_1\neq p_2$.  Then, given $\mu^\star$ as in \eqref{mu-star},
	there results that 
  \begin{equation}
  \F^\infty(\mu)=
      \begin{cases}
          \Eps_{p_1}(\mu)\quad \text{if}\quad (p_1-p_2)(\mu-\mu^\star)>0,\\ 
          \Eps_{p_2}(\mu)\quad \text{if}\quad (p_1-p_2)(\mu-\mu^\star)<0.
      \end{cases}
  \end{equation}
		\end{proposition}
	\begin{proof}
Let us introduce for every $\mu>0$ the function \[h(\mu):=\mathcal{E}_{p_1}(\mu,\Rd) -\mathcal{E}_{p_2}(\mu,\Rd)=-\rho_{p_1}\mu^{\f{2}{4-{p_1}}} +\rho_{p_2}\mu^{\f{2}{4-p_2}},\]
and observe that 
\begin{equation}
\lim_{\mu\to 0}h(\mu)=0,\quad \lim_{\mu\to+\infty}h(\mu)=\begin{cases}
-\infty, \quad \text{if}\quad p_1>p_2,\\
+\infty,\quad \text{if}\quad p_1<p_2,
\end{cases}
\end{equation}
By direct computation, 
\[
h'(\mu)=-\f{2\rho_{p_1}}{4-p_1}\mu^{\f{p_1-2}{4-p_1}} + \f{2\rho_{p_2}}{4-p_2}\mu^{\f{p_2-2}{4-{p_2}}},
\]
hence there exists $\widetilde{\mu}>0$ such that $h'(\mu)<0$ if and only if $(p_1-p_2)(\mu-\widetilde{\mu})<0$. This entails that $h(\mu)<0$ if and only if $(p_1-p_2)(\mu-\mu^\star)>0$, with $\mu^\star$ being the only value for which $h(\mu)=0$: the thesis follows from the definition of $h$.
	\end{proof}
	
	\subsection{Modified Gagliardo-Nirenberg inequalities}

 \label{subsec:GN}
 
	We recall now some Gagliardo-Nirenberg inequalities that will be useful in the following. Let us start from the classical Gagliardo-Nirenberg inequality on $\R^2$: for every $p>2$ \cite[Theorem 1.3.7]{cazenave} there exists $K_p$ such that
	\begin{equation}
		\label{eq:GNstandard}
		\norma{u}^{p}_{L^{p}(\R^2)} \leq K_p\,\norma{\nabla u}^{p-2}_{L^2(\R^2)}\norma{u}^2_{L^2(\R^2)}\qquad \forall u \in H^1(\R^2).
	\end{equation}
Let us now present some modified Gagliardo-Nirenberg inequalities proved in \cite{ABCT,CFN} valid for functions in 
\[
\D= \{u\in L^2(\R^2):\exists\, q\in\C, \lambda>0 : u-q\mathcal{G}_{\lambda}:=\phi_{\lambda}\in H^1(\R^2)\}.
\]
First of all, let us observe that for $\sigma \in \R$ and $\lambda>4e^{-4\pi\sigma-2\gamma}$ the space $\D$ can be equipped with the norm 
	\begin{equation}
		\label{norma}
	\NN_{\sigma,\lambda}(u) :=\bigl(Q_{\sigma}(u) + \lambda\norma{u}^2_2\bigl)^{\frac{1}{2}}. 
	\end{equation}
   The presence of such a limitation on the values of $\sigma$ for the definition \eqref{norma} is due to the fact that the quadratic form $Q_\sigma$ is not positive and 
   \begin{equation*}
   \inf_{v\in \D\setminus\{0\}}\frac{Q_\sigma(v)}{\|v\|_{L^2(\Rd)}^2}=-4e^{-4\pi\sigma-2\gamma}.
   \end{equation*}
   On the one hand,
	the authors in \cite{CFN} proved that there exists $C_p>0$ such that 
	\begin{equation}
		\label{eq:GN-noja}
		\norma{u}^{p}_{L^{p}(\R^2)}\leq C_p \NN_{\sigma,\lambda}(u)^{p-2} \norma{u}^2_{L^2(\R^2)},\qquad \forall \,u\in\D.
	\end{equation} 
 
 On the other hand, in \cite{ABCT} the authors proved the existence of a constant $M_p>0$ such that every function $u=\phi_{\la(u)}+q\G_{\la(u)}\in \D\setminus H^1(\Rd)$, with $\la(u)=|q|^2/{\|u\|_2^2}$, satisfies
		\begin{equation}
			\label{eq:GN-abct}
			\norma{u}^{p}_{L^{p}(\R^2)} \leq M_p (\norma{\nabla \phi_{\lambda(u)}}^{p-2}_{L^2(\R^2)}+ \abs{q}^{p-2})\norma{u}^2_{L^2(\R^2)}.
		\end{equation}
Let us observe that, differently from \eqref{eq:GN-noja}, in \eqref{eq:GN-abct} the authors take advantage of a specific representation of the functions $u\in \D\setminus H^1(\Rd)$, in which the parameter $\la$ depends on the function $u$ itself.


		\section{Proof of Theorem \ref{thm:existence}}

  \label{sec:prfthm12}
	This section is entirely devoted to the proof of Theorem \ref{thm:existence}.
 
 \begin{proof}[Proof of Theorem \ref{thm:existence}]
  Let $U_{n}=(u_{1,n},u_{2,n})$  be a minimizing sequence for \eqref{eq:F} at mass $\mu$, i.e.
  \begin{equation*}
  \|U_n\|_{2}^2=\mu\quad \forall\,n\quad\text{and}\quad F(U_n)\to \F(\mu)\quad \text{as}\quad n\to+\infty,
  \end{equation*}
  and $q_{i,n}$ be the charge associated with $u_{i,n}$ for $i=1,2$. Let $\sigma=\min\{\sigma_1,\sigma_2\}$, fix $\la=4e^{4\pi(2\beta-\sigma)-2\gamma}$ and consider the decomposition $u_{i,n}=\phi_{i,n}+q_{i,n}\G_\la$ for $i=1,2$: when not necessary, we omit the parameter $\la$ to lighten the notation.

We divide the proof in three steps.

  \emph{Step 1. There exist $\phi_i\in H^1(\Rd)$ and $q_i\in \C$ such that, up to subsequences, \begin{equation*}
    \phi_{i,n}\deb \phi_i \quad \text{in}\quad H^1(\Rd),\quad q_{i,n}\to q_i \quad \text{in}\quad \C. 
\end{equation*}}
  By using that $\|U_n\|_{2}^2=\mu$, \eqref{eq:GN-noja} and \eqref{norma} and observing that 
 \[
 (\sigma+\theta_\la)|q_{i,n}|^2\leq (\sigma_i+\theta_\la)|q_{i,n}|^2\leq \NN_{\sigma_i}(u_{i,n})^2,\quad \forall\,n\in\N,\quad i=1,2,
 \] 
 one obtains that
 \begin{equation}
 \label{bound-Fun-partial}
 \begin{split}
  F(U_n)\geq & \frac{1}{2}\NN_{\sigma_1}(u_{1,n})^2+ \frac{1}{2}\NN_{\sigma_2}(u_{2,n})^2-\frac{\beta}{\sigma+\theta_\la} \NN_{\sigma_1}(u_{1,n})\NN_{\sigma_2}(u_{2,n})\\
  &- \frac{\lambda}{2}\mu - \frac{C_{p_1}\,\mu}{p_1}\NN_{\sigma_1}(u_{1,n})^{p_1-2} - \frac{C_{p_2}\,\mu}{p_2}\NN_{\sigma_2}(u_{2,n})^{p_2-2} 
  \end{split}
 \end{equation}
 Since $\sigma+\theta_\la=2\beta$ according to our choice of $\la$, \eqref{bound-Fun-partial} entails  
 \begin{equation}
 \label{bound-Fun}
 F(U_n)\geq \frac{1}{4}\NN_{\sigma_1}(u_{1,n})^2+ \frac{1}{4}\NN_{\sigma_2}(u_{2,n})^2-\frac{C_{p_1}\,\mu}{p_1}\NN_{\sigma_1}(u_{1,n})^{p_1-2} - \frac{C_{p_2}\,\mu}{p_2}\NN_{\sigma_2}(u_{2,n})^{p_2-2} -\frac{\la}{2}\mu.    
 \end{equation}
Since $F(U_n)<0$ and $p_i<4$, by \eqref{bound-Fun} it follows that  $\NN_{\sigma_i}(u_{i,n})$ are equibounded for $i=1,2$, in particular $\phi_{i,n}$ are equibounded in $H^1(\Rd)$ and $q_{i,n}$ are equibounded in $\C$, hence by Banach-Alaoglu-Bourbaki theorem the thesis of Step 1 follows.\\

\emph{Step 2. It holds that $\F(\mu)<\F^\infty(\mu)$.}
By lemma \ref{lem:F-inf} and \cite[Proposition 3.2]{ABCT}, we have that
		\begin{equation}
  \label{eq:Finf>F1F2}
			\F^{\infty}(\mu)=\min\{\mathcal{E}_{p_1}(\mu),\mathcal{E}_{p_2}(\mu)\}> \min\{\F_{1}(\mu), \F_{2}(\mu)\}.
		\end{equation}
		
  Moreover, let us notice that 
  \begin{equation}
  \label{eq:inf-Fmu-moduli}
  \inf_{U\in\La_{\mu}}F(U)= \inf_{U\in\La_{\mu}}\{F_{1}(u) + F_{2}(v) - \beta \abs{q_{1}}\,\abs{q_{2}}\}\quad\forall\,\mu>0.
  \end{equation}
  Indeed, on the one hand it holds for every $U=(u,v)\in \La_\mu$ that
  $F(U)\geq F_1(u)+F_2(v)-\beta|q_1||q_2|$. On the other hand, if $u=\phi_1+q_1\G_\la$ and $v=\phi_2+q_2\G_\la$, then it is possible to find a proper phase factor $e^{i\theta}$ such that the function $\widetilde{U}=(u,e^{i\theta} v)$ satisfies
  \begin{equation*}
     \Re\left\{q_1\overline{e^{i\theta}q_2}\right\}=|q_1||q_2|,
  \end{equation*}
  hence passing to the infimum the equality is achieved. 

By \eqref{eq:inf-Fmu-moduli}, it follows that
		\begin{equation}
  \label{eq:Fmu-beta0}
  \begin{split}
	\F(\mu)=\inf_{U\in\La_{\mu}}F(U)&= \inf_{U\in\La_{\mu}}\{F_{1}(u) + F_{2}(v) - \beta \abs{q_{1}}\,\abs{q_{2}}\}\\
    &\leq \inf_{U\in\La_{\mu}}\{F_{1}(u) + F_{2}(v)\}\\
    &=\inf_{t\in[0,\mu]}\{\F_1(t)+\F_2(\mu-t)\}.
     \end{split}
		\end{equation}
	Since $\F_{1}$ and $\F_{2}$ are strictly concave functions of the mass by \cite[Lemma 4.2]{ABCT-24}, then also the function $\F_1(t)+\F_2(\mu-t)$ turns out to be a strictly concave function of $t$ and attains its minimum at $t=0$ or $t=\mu$, so that
		\begin{equation}
			\label{eq:inf-t-F1F2}
			\inf_{t\in[0,\mu]}\{\F_1(t)+\F_2(\mu-t)\}=\min\{\F_{1}(\mu), \F_{2}(\mu)\}.
		\end{equation}
  By combining \eqref{eq:Finf>F1F2}, \eqref{eq:Fmu-beta0} and \eqref{eq:inf-t-F1F2}, we conclude the proof of Step 2.\\
  
  \emph{Step 3. The function $U=(u_1,u_2)=(\phi_1+q_1\G_\la, \phi_2+q_2\G_\la)$ is a ground state of \eqref{eq:F} at mass $\mu$.} By Step 1, we deduce that $u_{i,n}\deb u_i$ weakly in $L^2(\Rd)$ since $\phi_{i,n}\deb \phi_i$ weakly in $L^2(\Rd)$ and $q_{i,n}\G_\la\to q_i\G_\la$ in $L^2(\Rd)$. 
  
  Set now  $m=\|U\|_{L^{2}(\I)}^{2}:=\|u_1\|_{L^{2}(\Rd)}^{2}+\|u_2\|_{L^{2}(\Rd)}^{2}$. First, we observe that by weak lower semicontinuity of the norms
		\begin{equation*}
			\|u_1\|_{L^{2}(\R^2)}^{2}+\|u_2\|_{L^{2}(\Rd)}^{2}\leq \liminf_{n\to+\infty}\|u_{1,n}\|_{L^{2}(\Rd)}^{2}+\liminf_{n\to+\infty}\|u_{2,n}\|_{L^{2}(\Rd)}^{2}\leq \liminf_{n\to+\infty}\|U_{n}\|_{L^{2}(\I)}^{2},
		\end{equation*}
  thus $m\leq \mu$. 
  
  Suppose first that $m=0$, i.e. $u\equiv 0$ and $v\equiv 0$. In particular, this entails that $q_{1,n}\to 0$ and $q_{2,n}\to 0$ as $n\to +\infty$. Therefore, relying on Step 2 we deduce that for $n$ sufficiently large 
  \begin{align}
			\mathcal{F}^{\infty}(\mu)>\F(\mu)=& \lim_{n\to +\infty}\left\{ F_{1}(u_{1,n}) + F_{2}(u_{2,n}) -\beta \,\abs{q_{1,n}}\,\abs{q_{2,n}}\right\}\\
			=& \lim_{n\to +\infty} \left\{E_{p_1}(\phi_{1,n}) + E_{p_2}(\phi_{2,n})\right\}\\
			\geq&
			\lim_{n\to +\infty}\left\{\mathcal{E}_{p_1}(\norma{\phi_{1,n}}^2_2) +  \mathcal{E}_{p_2}(\norma{\phi_{2,n}}_2^2)\right\}\\
   =& \lim_{n\to +\infty}\left\{\mathcal{E}_{p_1}(\norma{u_{1,n}}^2_2) +  \mathcal{E}_{p_2}(\norma{u_{2,n}}_2^2)\right\}\\
			\geq &\, \inf\{\mathcal{E}_{p_1}(\mu_1) \,+\, \mathcal{E}_{p_2}(\mu_2):\mu_1 +\mu_2=\mu\}=\mathcal{F}^{\infty}(\mu),
		\end{align}
thus $m\neq 0$. 

Suppose now that $0<m<\mu$. Note that, since $u_{i,n}\deb u_i$ in $L^{2}(\Rd)$ for $i=1,2$, then
		\begin{equation}
			\label{eq:mass-m-mu}
			\|U_{n}-U\|_{L^{2}(\I)}^{2}=\norma{U_n}_2^2 - \norma{U}_2^2 +  o(1) =\mu-m+o(1),\quad\text{as}\quad n\to+\infty.
		\end{equation}
	
	On the one hand, since $p_i>2$ for $i=1,2$, $\f{\mu}{\|U_{n}-U\|_{2}^{2}}=\frac{\mu}{\mu-m}+o(1)>1$ as $n\to+\infty$ by \eqref{eq:mass-m-mu} and 
		$\norma{\sqrt{\f{\mu}{\|U_{n}-U\|_{2}^{2}}}  (U_{n}-U)}_2^2 =\mu$, there results that
		\begin{align}
			\F(\mu)&\leq F\left(\sqrt{\f{\mu}{\|U_{n}-U\|_{2}^{2}}} (U_{n}-U)\right)\\
   &=\frac{1}{2}\f{\mu}{\|U_{n}-U\|_{2}^{2}}Q(U_n-U)-\sum_{i=1,2}\frac{1}{p_i} \left(\f{\mu}{\|U_{n}-U\|_{2}^{2}}\right)^{\f{p_i}{2}}\norma{u_{i,n} -u_i}_{p_i}^{p_i}\\
			&<\f{\mu}{\|U_{n}-U\|_{L^{2}(\I)}^{2}}F(U_{n}-U), 
		\end{align}
		hence
		\begin{equation}
			\label{eq:FUn-U}
			\liminf_{n\to+\infty} F(U_{n}-U)\geq \f{\mu-m}{\mu}\F(\mu). 
		\end{equation}
		On the other hand, by similar computations
		\begin{equation*}
			\F(\mu)\leq F\left(\sqrt{\f{\mu}{\|U\|_{2}^{2}}}U\right)<\f{\mu}{\|U\|_{2}^{2}}F(U)
		\end{equation*}
	 thus
		\begin{equation}
			\label{eq:FU>mmu}
			F(U)>\f{m}{\mu}\F(\mu).
		\end{equation}
		
		In addition, since 
 $u_{i,n}\deb u_i$ in $L^2(\Rd)$, $\phi_{i,n}\deb \phi_{i}$ in $H^{1}(\Rd)$ and $q_{i,n}\to q_i$ in $\C$, then 
		\begin{align}
			Q(U_{n}-U)&=Q(U_{n})-Q(U)+o(1),\quad\text{as}\quad n\to +\infty,
		\end{align}
		while by Brezis-Lieb lemma \cite{BL-83} applied to $u_{i,n}$ we get
		\begin{equation*}
			\|u_{i,n}\|_{L^{p_i}(\Rd)}^{p_i}=\|u_{i,n}-u\|_{L^{p_i}(\Rd)}^{p_i}+\|u\|_{L^{p_i}(\Rd)}^{p_i}+o(1),\quad\text{as}\quad n\to+\infty.
		\end{equation*}
        Thus
		\begin{equation}
			\label{eq:BrezisLieb}
			F(U_{n})= F(U_{n}-U)+F(U)+o(1),\quad\text{as}\quad n\to+\infty.
		\end{equation}
		Combining \eqref{eq:FUn-U}, \eqref{eq:FU>mmu} and \eqref{eq:BrezisLieb}, it follows that
		\begin{equation*}
			\F(\mu)=\liminf_{n\to+\infty} F(U_{n})=\liminf_{n\to+\infty} F(U_{n}-U)+F(U)>\f{\mu-m}{\mu}\F(\mu)\,+\,\f{m}{\mu}\F(\mu)=\F(\mu),
		\end{equation*}
		which, being a contradiction, forces $m=\mu$. In particular, $U\in \La_{\mu}$ and thus $u_{i,n}\to u_i$ in $L^{2}(\Rd)$ and $\phi_{i,n}\to \phi_{i}$ in $L^{2}(\Rd)$. 
        
        In order to conclude that $U$ is a ground state, we are left to prove that
		\begin{equation}
			\label{Fsigma-lsc}
			F(U)\leq \liminf_{n\to+\infty} F(U_{n})=\F(\mu),
		\end{equation}
  which reduces to prove that $u_{i,n}\to u_i$ in $L^{p_i}(\Rd)$ for $i=1,2$. By \eqref{eq:GN-noja}, we obtain that
		\begin{equation*}
			\|u_{i,n}-u_i\|_{L^{p_i}(\Rd)}^{p_i}\leq C_{p_i}\NN_{\sigma_i}(u_{i,n}-u_i)^{p_i-2}\|u_{i,n}-u_i\|_{L^{2}(\Rd)}^{2}\to 0,\quad \text{as}\quad n \to +\infty,
		\end{equation*}
		which concludes the proof since $\NN_{\sigma_i}(u_{i,n}-u_i)$ is equibounded and $u_{i,n}\to u_i$ in $L^{2}(\Rd)$.
 \end{proof}

\begin{remark}
 Let us stress that the choice of $\la$ in the proof of Theorem \ref{thm:existence} is not the only possible one: it is sufficient to choose $\la>0$ such that $\sigma+\theta_\la>2\beta$ and all the arguments in the proof can be repeated. 
\end{remark}

		\section{Properties of ground states: proof of Theorem \ref{thm:charact}}

  \label{sec:prfthm13}
		This section is dedicated to the proof of the Theorem \ref{thm:charact}. The first result deals with the case $\beta>0$ and establishes that $u_i\not \equiv 0$ and $q_i>0$ for $i=1,2$, up to the multiplication by a phase factor.
  
  \begin{proposition}
			\label{prop:qi>0}
			Let $p_i\in (2,4)$, $\sigma_i\in\R$ for $i=1,2$, $\beta>0$ and $\mu>0$. Let $U=(u_1,u_2)\in \Lambda_\mu$ be a ground state of $\F$ at mass $\mu$, then $u_i\not\equiv 0$ and $q_i>0$ for $i=1,2$, up to a multiplication by a phase factor. Moreover, for $i=1,2$ and $\la_i>0$ there results that $\phi_{i}:=u_i-q_i\G_{\la_i}\not\equiv 0$.
		\end{proposition}
		
		\begin{proof}
			Let $U=(u_1,u_2)$ be a ground state of $F$ at mass $\mu$ and $\la_i>0$ for $i=1,2$, so that $u_i=\phi_i+q_i\G_{\la_i}$. Since $U$ solves \eqref{eq:stat-detailed}, there results that
			
   \begin{equation*}
       \begin{cases}
					\phi_{1}(0)&=  (\sigma_1 + \theta_{\lambda_1})q_1-\beta\,q_2,\\
					\phi_{2}(0)&=-\beta\,q_1 + (\sigma_2 + \theta_{\lambda_2})q_2.
			\end{cases}
    \end{equation*}
    
			\emph{Step 1. $u_i\not\equiv 0$ for $i=1,2$.} Let us first prove that $u_1\not\equiv 0$. Suppose by contradiction that $u_1\equiv 0$. Therefore, since $\beta>0$, it follows that $q_2=0$, hence $u_2\in H^1(\R^2)$.
		Therefore, since $U$ is a ground state of $F$ at mass $\mu$, by Step 2 in Theorem \ref{thm:existence} it holds that 
			\begin{equation*}
			 \F(\mu)=F(U)= F_{2}(u_2)= \Eps_{p_2}(\mu)\geq\mathcal{F}^{\infty}(\mu)>\F(\mu),    \end{equation*}
			which is a contradiction. Analogously, one can prove that $u_2\not\equiv 0$.\\
   
   \emph{Step 2. $q_i\neq 0$ for $i=1,2$.}
	Let us first prove that $q_1\neq 0.$ If we suppose by contradiction that $q_1=0$, then 
			\[
			F(U)= E_{p_1}(u_1) + F_{2}(u_2)=\inf_{(u_1,u_2)\in \La_\mu} \{E_{p_1}(u_1)+F_2(u_2)\}.
			\]
			In particular, since the problems on the two planes are decoupled in this case, one can argue as in the proof of Lemma \ref{lem:F-inf} and prove that
			\begin{equation}
            \label{eq:inf-F-m-mu}
			F(U)=\inf_{(u_1,u_2)\in \La_\mu} \{E_{p_1}(u_1)+F_2(u_2)\}= \inf_{m\in[0,\mu]} \{\mathcal{E}_{p_1} (m) + \F_{2}(\mu-m) \}.    
			\end{equation}

			Since both $\mathcal{E}_{p_1}$ and $\F_{2}$ are strictly concave functions of the mass, then also the function $m\mapsto\mathcal{E}_{p_1} (m) + \F_{2}(\mu-m)$ is strictly concave in $[0,\mu]$ and attains its minimum at $m=0$ or $m=\mu$.  If one considers $U=(u_1,u_2)$ with $u_i\not\equiv 0$ for $i=1,2$, then $U$ cannot realize the infimum in \eqref{eq:inf-F-m-mu}: indeed, if this is the case, then by strict concavity of the function $m\mapsto\mathcal{E}_{p_1} (m) + \F_{2}(\mu-m)$ there exists a function $\widetilde{U}$ with either $u_1\equiv 0$ or $u_2\equiv 0$ with energy strictly less than the energy of $U$, contradicting the hypothesis that $U$ is a ground state of $F$. Therefore, it holds that either $u_1\equiv 0$ or $u_2\equiv 0$, but this is in contradiction with Step 1. In the same way, it is possible to prove that $q_2\neq 0$, concluding the proof of Step 2.\\

   \emph{Step 3. $\phi_i\not \equiv 0$ for $i=1,2$.}
			Let us prove first that $\phi_{1}\not\equiv 0$. Indeed, if we suppose by contradiction that $\phi_{1}\equiv 0$, then by the first equation in \eqref{eq:stat-detailed} we have that 
			\begin{align}
			  &\omega-\la_1=|q_1|^{p_1-2}\mathcal{G}_{\lambda_{1}}^{p_1-2}(x),\qquad \, \forall\, x\in\R^2\setminus\{0\},
			\end{align}
			but this is a contradiction since the left-hand side is constant, while the right-hand side is a multiple of $\mathcal{G}_{\lambda_1}$, hence $\phi_{1}\not\equiv 0.$ Analogously, one can show that $\phi_{2}\not\equiv 0.$\\
   
\emph{Step 4. $q_i>0$ for $i=1,2$, up to a phase factor.} First of all, let us observe that $q_1=|q_1|e^{i\varphi_1}$ and $q_2=|q_2|e^{i\varphi_2}$ share the same phase, i.e. $e^{i\varphi_1}=e^{i\varphi_2}$. Indeed, suppose by contradiction that $e^{i\varphi_1}\neq e^{i\varphi_2}$. Then the function $U_\varphi=(u_1,e^{i(\varphi_1-\varphi_2)}u_2)$ satisfies $\|U_\varphi\|_2^2=\|U\|_2^2$ and
\begin{equation*}
\begin{split}
F(U_\varphi) &= F_1(u_1)+F_2\left(u_2 e^{i(\varphi_1-\varphi_2)}\right)-\beta \Re\{q_1\overline{q_2}e^{-i(\varphi_1-\varphi_2)}\}\\
&=F_1(u_1)+F_2(u_2)-\beta|q_1||q_2|\\
&<F_1(u_1)+F_2(u_2)-\beta\Re\{q_1\overline{q_2}\}=F(U),
\end{split}
\end{equation*}
which contradicts the fact that $U$ is a ground state of $F$ at mass $\mu$. Now, if $q_i$ are not positive, i.e. $q_i=|q_i|e^{i\varphi}$ with $e^{i\varphi}\neq 1$, then it is sufficient to observe that the function $e^{-i\varphi}U$ has positive charges $q_i$ and satisfies $\|e^{-i\varphi}U\|_2^2=\mu$ and $F\left(e^{-i\varphi}U\right)=F(U)$, thus without loss of generality we can suppose that $q_i>0$ for $i=1,2$. 
\end{proof}
		In order to prove Theorem \ref{thm:charact}, we follow the strategy introduced in \cite{ABCT} and \cite{ABCT-24}, where the properties of ground states are proved switching to alternative minimization problems and relying on the connection between them and the original problem itself. In particular, we first introduce the action functional $S_{\omega}$. Given $\omega\in\R$, the action functional at frequency $\omega$ is the functional $S_{\omega}:\La\to \R$ defined by
		\begin{equation}
			\label{eq:action}
			S_{\omega}(U)=F(U) + \frac{\omega}{2}\norma{U}_{L^2(\I)}^2,
		\end{equation}
        where $L^2(\mathcal{I}):=L^2(\Rd)\oplus L^2(\Rd).$  Denoted by $I_{\omega}:\La\to \R$ the functional
		\begin{equation}
			\label{eq:I-om}
			I_{\omega}(U)=Q(U) + \omega \norma{U}_{L^2(\mathcal{I})}^2 - \norma{u_1}_{L^{p_1}(\R^2)}^{p_1} - \norma{u_2}_{L^{p_2}(\R^2)}^{p_2},
   \end{equation}
   one defines the Nehari's manifold at frequency $\omega$ as 
		\[
		N_{\omega}:=\{U\in \La\setminus\{0\} : I_{\omega}(U)=0\},
		\]
		
		Let us also introduce
		\begin{equation}
			Q_{\omega}(U):= Q(U) + \omega \norma{U}_{L^2(\mathcal{I})}^2.
		\end{equation}
		
		\begin{definition}
			A function $U$ belonging to the Nehari manifold $N_{\omega}$ and satisfying 
			\[
			S_{\omega}(U) = d(\omega):= \inf_{V\in N_{\omega}}S_{\omega}(V)
			\]
			is called a minimizer of the action at frequency $\omega$.
		\end{definition}
		Let us highlight that also minimizers of the action at frequency $\omega$ satisfy \eqref{eq:stat}. In the next lemma, we state a connection between ground states at fixed mass and minimizers of the action at fixed frequency. The proof is analogous to the case of the standard NLS \cite{DST-23, JL-22}. Before stating it, let us define by 
 \begin{equation}
 \omega^\star:=-\inf_{U\in \Lambda\setminus\{0\}}\frac{Q(U)}{\|U\|_2^2}.   
 \end{equation}

		\begin{lemma}
			\label{lem:colleg}
			Let $\mu>0$. If $U$ is a ground state of $F$ at mass $\mu$, then $U$ is a minimizer of the action at frequency
			\begin{equation}
				\label{eq:omega}
				\omega=\mu^{-1}\Bigl(\norma{u_1}^{p_1}_{L^{p_1}(\R^2)} + \norma{u_2}^{p_2}_{L^{p_2}(\R^2)} -Q(U) \Bigl) > \omega^\star.
			\end{equation}
		\end{lemma} 
		\begin{proof} 
			Let $U=(u_1,u_2)$ be a ground state of $F$ at mass $\mu$ and $\omega>0$ be the associated Lagrange multiplier, given by \eqref{eq:omega}. Assume by contradiction that there exists $V=(v_1,v_2)\in N_{\omega}$ such that $S_{\omega}(V)<S_{\omega}(U)$ and let $\tau>0$  be such that $||\tau V||_{L^2(\R^2)}^2=\mu$. Then
			\begin{align}
				S_{\omega}(\tau V)=
				\frac{\tau^2}{2} Q_{\omega}(V) - \frac{\tau^{p_1}}{p_1} \norma{v_1}_{L^{p_1}(\R^2)}^{p_1}- \frac{\tau^{p_2}}{p_2} \norma{v_2}_{L^{p_2}(\R^2)}^{p_2}.
			\end{align}
			Computing the derivative with respect to $\tau$ and using that $V\in N_{\omega}$, we get
			\begin{equation*}
				\frac{d}{d\tau} S_{\omega}(\tau V)=\tau\left[(1-\tau^{p_1-2})\norma{v_1}_{p_1}^{p_1} + (1-\tau^{p_2-2})\norma{v_2}_{p_2}^{p_2}\right],
			\end{equation*}
			which is greater than equal to zero if and only if $0<\tau\leq 1$. Hence $S_{\omega}(\tau V)\leq S_{\omega}(V)<S_{\omega}(U)$ for every $\tau>0$ which, combined with $||\tau V||_{2}^2=\norma{U}_{2}^2=\mu$, entails that $F(\tau V)<F(U)$. Since this contradicts the fact that $U$ is a ground state of $F$ at mass $\mu$, it follows that $U$ is a minimizer of the action at frequency $\omega$.
   
  \noindent Finally, \eqref{eq:omega} follows from
			\begin{align}
				-\omega &= \frac{Q(U) - \norma{u_1}^{p_1}_{p_1} - \norma{u_2}^{p_2}_{p_2}}{\mu}=\frac{2{F(U)} - \frac{p_1-2}{p_1}\norma{u_1}_{p_1}^{p_1} - \frac{p_2-2}{p_2}\norma{u_2}_{p_2}^{p_2}}{\mu}\\ 
				&<2\frac{\F(\mu)}{\mu}< \inf_{U\in \La_\mu}\frac{Q(U)}{\mu}=- \omega^\star.
			\end{align}
		\end{proof}
		Let us now introduce the functionals $\widetilde{S}, A_\om, B_\om:\La\to\R$, defined as
		\begin{align}
			\widetilde{S}(U)& := \frac{p_1-2}{2p_1} \norma{u_1}^{p_1}_{L^{p_1}(\R^2)} + \frac{p_2-2}{2p_2} \norma{u_2}^{p_2}_{L^{p_2}(\R^2)},\\
			A_{\omega}(U)&:=\frac{p_1-2}{2p_1} Q_{\omega}(U) + \frac{p_2-p_1}{p_1p_2} \norma{u_2}^{p_2}_{L^{p_2}(\R^2)},\\
		B_{\omega}(U)&:= \frac{p_2-2}{2p_2} Q_{\omega}(U) + \frac{p_1-p_2}{p_1p_2} \norma{u_1}^{p_1}_{L^{p_1}(\R^2)} ,
		\end{align}
		and observe that, by using \eqref{eq:I-om},  the action can be rewritten alternatively as 
		\[
		S_{\omega}(U)=\frac{1}{2} I_{\omega}(U) + \widetilde{S}(U) =\frac{1}{p_1}I_{\omega}(U) + A_{\omega}(U)=\frac{1}{p_2}I_{\omega}(U) + B_{\omega}(U),
		\]
		so that 
		\[
		d(\omega)=\inf_{U\in N_{\omega}}\widetilde{S}(U)=\inf_{U\in N_{\omega}} A_{\omega}(U)=\inf_{U\in N_{\omega}} B_{\omega}(U).
		\]
    
    In the following lemmas, we provide other equivalent formulations for $d(\omega)$ involving the three functionals $\widetilde{S}$, $A_\omega$ and $B_\omega$ subject to other constraints. For the proofs, we refer to \cite[Lemma 5.6-5.7]{ABCT-24}, where the same results were proved in an analogous setting.
		\begin{lemma}
			Let $p_i>2$ for $i=1,2$ and $\omega>\omega^\star$. Then, denoted with 
   \[
			\widetilde{N}_{\omega}:=\{V\in \La\setminus\{0\},\, I_{\omega}(V)\leq 0\},
			\]
   there results that
			\[
			d(\omega)=\inf_{V\in \widetilde{N_{\omega}}}\widetilde{S}(V).
			\]
			
			Moreover, for any function $U\in \La\setminus\{0\}$, it holds that
			\[	\begin{cases}
				\widetilde{S}(U)=d(\omega)\\
				I_{\omega}(U)\leq 0
			\end{cases}
			\iff\qquad
			\begin{cases}
				S_{\omega}(U)=d(\omega)\\
				I_{\omega}(U)=0.
			\end{cases}
			\]
		\end{lemma} 
		\begin{lemma}
			\label{lem:equiv-Aom}
			Let $p_i>2$ for $i=1,2$ and  $\omega>\omega^\star$. Then, denoted with
			\[
			\Pi_\om:=\{V\in \La: \widetilde{S}(V)=d(\omega)\},
			\]
			there results that
			\[
			d(\omega)=\inf_{V\in \Pi_{\omega}} A_{\omega}(V)=\inf_{V\in \Pi_{\omega}} B_{\omega}(V).
			\]
			In particular, for every $U\in \La\setminus\{0\}$, it holds that 
			\[\begin{cases}
				A_{\omega}(U)=d(\omega)\\
				U\in \Pi_{\omega}
			\end{cases}
			\iff\qquad
			\begin{cases}
				B_{\omega}(U)=d(\omega)\\
				U\in \Pi_{\omega}
			\end{cases}
                \iff\qquad
			\begin{cases}
				S_{\omega}(U)=d(\omega)\\
				I_{\omega}(U)=0.
			\end{cases}
			\]
		\end{lemma}

We are now ready to prove Theorem \ref{thm:charact}.
  
\begin{proof}[Proof of Theorem \ref{thm:charact}]

Let us focus first on the proof of \eqref{eq:Fmu=F12mu} and \eqref{eq:Fmu<F12mu}, when $\beta=0$ and $\beta>0$ respectively. In particular, if $\beta=0$, then  
	\[
	\F(\mu)=\inf_{U\in\La_\mu} \left\{F_1(u) + F_2(v)\right\}=\inf_{t\in[0,\mu]} \left\{\F_{1}(t) + \F_{2}(\mu-t)\right\},
    \]
hence by applying \eqref{eq:inf-t-F1F2} we deduce \eqref{eq:Fmu=F12mu} and the fact that every ground state is either of the form $U=(u_1,0)$ or of the form $U=(0,u_2)$, depending on which is the minimum between $\F_1(\mu)$ and $\F_2(\mu)$: this entails in particular $(i)$, $(ii)$ and $(iii)$ and the properties of the non-zero component.  

If instead $\beta>0$, then
\[
	\F(\mu)=\inf_{U\in\La_\mu} \left\{F_1(u) + F_2(v)-\beta|q_1||q_2|\}\right\}\leq\inf_{t\in[0,\mu]} \left\{\F_{1}(t) + \F_{2}(\mu-t)\right\},
    \]
thus by \eqref{eq:inf-t-F1F2} we get that
\begin{equation*}
    \F(\mu)\leq \min\{\F_{1}(\mu), \F_{2}(\mu)\}.
\end{equation*}

In order to prove \eqref{eq:Fmu<F12mu}, it is sufficient to exclude that $\F(\mu)=\min\{\F_1(\mu), \F_2(\mu)\}$. If we suppose by contradiction that $\F(\mu)=\min\{\F_1(\mu), \F_2(\mu)\}$, then there exist also ground states either of the form $U=(u_1,0)$ or of the form $U=(0,u_2)$. But this contradicts Proposition \ref{prop:qi>0}, hence \eqref{eq:Fmu<F12mu} holds. 

Fix now $\beta>0$ and let $U=(u_1,u_2)$ be a ground state of $F$ at mass $\mu$: its existence is ensured by Theorem \ref{thm:existence}. By proposition \ref{prop:qi>0}, we deduce that the associated charges $q_i$ satisfy $q_i>0$ for $i=1,2$, up to multiplication by a phase factor. Moreover, by Lemma \ref{lem:colleg} $U$ is also a minimizer of the action at frequency $\omega$, with $\omega>\omega^\star$. 

Let us now consider the decomposition of $u_i$ corresponding to the Lagrange multiplier $\omega$, namely $u_i=\phi_{i,\omega}+q_i\G_\omega$. Since the proof of the positivity and the radiality of $\phi_{i,\omega}$ and $u_i$ is quite long, we divide it into three steps. \\

\textit{Step 1. Positivity of $\phi_{i,\omega}$ and $u_i$}. We first prove that $\phi_{1,\omega}>0$. Let us start defining the two sets
			\[
			\Omega := \{x\in\R^2\setminus\{0\}:\phi_{1,\omega}(x)\neq 0\}
			\]
            and 
            \[
   \widehat{\Omega}:=\{x\in\Omega : \al(x)\neq 0\},
   \]
   where $\alpha:\Omega\to [0,2\pi)$ is such that $\phi_{1,\omega}(x)=e^{i\alpha(x)}|\phi_{1,\omega}(x)|$.
			To prove that $\phi_{1,\omega}>0$, we are reduced to prove that $\Omega=\Rd$ and $|\widehat{\Omega}|=0$. 
			
            We start proving that $|\widehat{\Omega}|=0$. Assume by contradiction that $\abs{\widehat{\Omega}}>0$ and introduce the function 
   \[
   U_\tau:=(u_{1,\tau},u_2),\quad \text{with}\quad  u_{1,\tau}:=\tau\abs{\phi_{1,\om}}+q_1\G_\om.
   \]
			 Let us observe that, since $\|\na\abs{\phi_{1,\om}}\|_2^2\leq \|\na\phi_{1,\om}\|_2^2$, then for every $\tau\in (-1,1)$
\begin{align}
			Q_{\omega}(U_{\tau})&=\tau^2\left(\norma{\nabla \abs{\phi_{1,\omega}}}_{L^2(\R^2)}^2 + \omega \norma{\phi_{1,\omega}}_2^2\right)  + \abs{q_1}^2(\sigma_1 + \theta_{\om})+Q_{\sigma_2}(u_2) + \omega \norma{u_2}_2^2-\beta\, q_1 \,q_2\\
   &<\norma{\nabla \abs{\phi_{1,\omega}}}_{L^2(\R^2)}^2 + \omega \norma{\phi_{1,\omega}}_2^2  + \abs{q_1}^2(\sigma_1 + \theta_{\om})+Q_{\sigma_2}(u_2) + \omega \norma{u_2}_2^2-\beta\, q_1 \,q_2\\
   &=Q_{\omega}(U_{1})\leq Q_\omega(U).
\end{align}
   which entails in particular that 
   \begin{equation}
   \label{contr:lem-pos}
    A_\omega(U_\tau)<A_\omega(U_1)\leq A_\omega(U)\quad\forall\,\tau\in(-1,1).
 \end{equation}
 
\noindent We claim (and prove in Step 2) that there exists $\overline{\tau}\in (-1,1)$ such that 
 \begin{equation}
\label{claim:lem-pos}
\|u_{1,\overline{\tau}}\|_{p_1}^{p_1}=\|u_1\|_{p_1}^{p_1}.
\end{equation}
In fact, if we prove \eqref{claim:lem-pos}, then $U_{\overline{\tau}}\in \Pi_\omega$ and $A_\omega(U_\tau)<A_\omega(U)$. On the other hand, $U$ is a minimizer of $A_\omega$ on $\Pi_\omega$ by Lemma \ref{lem:equiv-Aom}, but this generates a contradiction, ensuring that $|\widehat{\Omega}|=0$ and $\phi_{1,\omega}>0$ on $\Omega$. 
Therefore, it is left to prove that $\Omega=\Rd$, i.e. $\phi_{1,\omega}(x)>0$ for every $x\in\R^2$. First of all, $|\Omega|>0$ by Proposition \ref{prop:qi>0}. Then, since $u_1$ solves \eqref{eq:stat-detailed}, $\phi_{1,\omega}$ satisfies 
			\[
			(-\lap + \omega) \phi_{1,\omega} =\abs{u_1}^{p_1-2}(\phi_{1,\omega} + q_1\mathcal{G}_{\omega}).
			\]
			Therefore, as $\phi_{1,\omega}$ is non-negative, $(-\lap + \omega) \phi_{1,\omega}\geq0$ in $H^{-1}(\R^2)$ and $\phi_{1,\omega}\not\equiv 0$ by Proposition \ref{prop:qi>0}, then by the Strong Maximum Principle (e.g.,\cite[Theorem 3.1.2]{C-18}) $\phi_{1,\omega}>0$ on $\R^2$. To prove that $\phi_{2,\omega}(x)>0$ for every $x\in\R^2$, we repeat the same arguments using the functional $B_{\omega}$ instead of ${A}_{\omega}$. The positivity of $u_i$ follows from the positivity of $\phi_{i,\omega}$ and $q_i$.\\

            \textit{Step 2. Proof of the claim \eqref{claim:lem-pos}.}
First of all, let us observe that the function $\tau\mapsto \|u_{1,\tau}\|_{p_1}^{p_1}$ is continuous on $[-1,1]$. Moreover, since 
\begin{equation*}
				\abs{u_1(x)}^2=\abs{\phi_{1,\omega}(x)}^2 + q_1^2\mathcal{G}_{\omega}^2(x)+ 2\cos(\al) q_1 \abs{\phi_{1,\omega}(x)}\mathcal{G}_{\omega}(x)< \abs{u_{1,1}(x)}^2,\quad \forall x\in\widehat{\Omega},
			\end{equation*}
			and $\abs{\widehat{\Omega}}>0$, there results that  
	$\norma{u_1}_{p_1}^{p_1}<\norma{u_{1,1}}_{p_1}^{p_1}$. 
 In order to prove \eqref{claim:lem-pos}, we need to distinguish two cases. Suppose first that there exists $\widetilde{\Omega}$, with $\abs{\widetilde{\Omega}}>0$, such that $\alpha\neq \pi$ on $\widetilde{\Omega}$. In this case, we have that
			\begin{align}
			 \norma{u_{1,-1}}^{p_1}_{L^{p_1}(\Rd)}&=\norma{u_{1,-1}}^{p_1}_{L^{p_1}(\Rd\setminus{\widetilde{\Omega}})}+\norma{u_{1,-1}}^{p_1}_{L^{p_1}(\widetilde{\Omega})}\\
    &=\norma{u_{1}}^{p_1}_{L^{p_1}(\Rd\setminus{\widetilde{\Omega}})}+\int_{\widetilde{\Omega}}
			\Bigl \lvert \abs{\phi_{1,\omega}}^2 + q_1^2\mathcal{G}_{\omega}^2 - 2\, q_1 \abs{\phi_{1,\omega}}\mathcal{G}_{\omega} \Bigr \rvert^{\frac{p_1}{2}}\, dx\\
   &<\norma{u_{1}}^{p_1}_{L^{p_1}(\Rd\setminus{\widetilde{\Omega}})}+\int_{\widetilde{\Omega}}
			\Bigl \lvert \abs{\phi_{1,\omega}}^2 + q_1^2\mathcal{G}_{\omega}^2 +2 \cos(\alpha) q_1 \abs{\phi_{1,\omega}}\mathcal{G}_{\omega} \Bigr \rvert^{\frac{p_1}{2}}\, dx\\
   &=\norma{u_{1}}^{p_1}_{L^{p_1}(\Rd)}
			\end{align}
		hence, since $\norma{u_{1,-1}}^{p_1}_{L^{p_1}(\R^2)}< \norma{u_1}^{p_1}_{L^{p_1}(\R^2)} <\norma{u_{1,1}}^{p_1}_{L^{p_1}(\R^2)}$, there exists $\overline{\tau}\in (-1,1)$ such that \eqref{claim:lem-pos} holds.
   
			Suppose instead that  $\alpha(x)=\pi$ for a.e. $x\in \Omega$: this choice of $\alpha$ can be reformulated also as $u_1(x)=u_{1,-1}(x)$ for a.e. $x\in\Omega$ or 
 as $\phi_{1,\omega}=-|\phi_{1,\omega}|$ for a.e. $x\in\Omega$. Differently from the previous case, we cannot conclude immediately using the continuity of $\|u_{1,\tau}\|_{p_1}^{p_1}$ since $\|u_{1,-1}\|_{p_1}^{p_1}=\|u_{1}\|_{p_1}^{p_1}$. To solve this issue, we observe that for every $\tau\in [-1,1]$
			\begin{align}
				\left| \frac{d}{d\tau}\abs{u_{1,\tau}}^{p_1}\right|&= \left| \frac{d}{d\tau}(\tau^2\abs{\phi_{1,\omega}}^2+ q_1^2 \mathcal{G}_{\omega}^2 + 2\tau q_1 \abs{\phi_{1,\omega}}\mathcal{G}_{\omega})^{\frac{p_1}{2}} \right| \\
    &=\left|p_1\left|\phi_{1,\omega}\right|\left|u_{1,\tau}\right|^{p_1-2}u_{1,\tau}\right|=p_1\left|\phi_{1,\omega}\right|\left|u_{1,\tau}\right|^{p_1-1}
				\leq p_1\abs{\phi_{1,\omega}}\abs{u_{1,1}}^{p_1-1}\in L^1(\R^2),
			\end{align}
			hence by dominated convergence  
			\begin{equation}
				\label{eq:d-dtau}
    \begin{split}
     \frac{d}{d\tau}{\norma{u_{1,\tau}}^{p_
     1}_{L^{p_1}(\R^2)}}\Bigr|_{\tau=-1^+}  = \int_{\R^2}  \frac{d}{d\tau}\abs{u_{1,\tau}}^{p_1}\Bigr|_{\tau=-1}  \,dx  &=  \int_{\R^2}\, p_1\abs{\phi_{1,\omega}}\abs{u_{1,-1}}^{p_1-2} u_{1,-1}\,dx\\
     &=\int_{\R^2}\, p_1\abs{\phi_{1,\omega}}\abs{u_{1}}^{p_1-2} u_{1}\,dx.   
    \end{split}
        \end{equation}
   
   Since $U$ is a ground state and solves  \eqref{eq:stat-detailed} with $\la_1=\om$,  there results that
    \begin{equation}
\label{EL-plane}
\scal{\na \chi}{\na \phi_{1,\omega}}_{L^{2}(\Rd)}+\scal{\chi}{\omega \phi_{1,\omega}-|v|^{r-2}v}_{L^{2}(\Rd)}=0\quad\forall\, \chi\in H^1(\Rd).
\end{equation}  
By choosing $\chi=|\phi_{1,\omega}|$ in \eqref{EL-plane} and recalling that $\phi_{1,\omega}=-|\phi_{1,\omega}|$ on $\Omega$ (and consequently in $\Rd$), it follows that
			\begin{equation}
				\int_{\R^2} \abs{\phi_{1,\omega}(x)}\abs{u_1(x)}^{p_1-2}u_1(x)\, dx =- \int_{\Rd} \abs{\nabla\abs{\phi_{1,\omega}(x)}}^2 \,dx- \omega\int_{\R^2} \abs{\phi_{1,\omega}(x)}^2\, dx,
			\end{equation}
			thus
			\begin{equation}
			 \label{eq:der-tau<0}   
			\frac{d}{d\tau}\norma{u_{1,\tau}}^{p_1}_{L^{p_1}(\R^2)} \,\biggr\rvert_{\tau=-1^+}< 0.
			\end{equation}
   Therefore, by combining \eqref{eq:der-tau<0} and the fact that $\norma{u_{1,-1}}^{p_1}_{L^{p_1}(\R^2)}= \norma{u_1}^{p_1}_{L^{p_1}(\R^2)} <\norma{u_{1,1}}^{p_1}_{L^{p_1}(\R^2)}$, \eqref{claim:lem-pos} follows.\\
			
            \textit{Step 3. Radiality of $\phi_{i,\omega}$ and $u_i$}. As done in Step 1, we recall that by lemma \ref{lem:equiv-Aom} $U$ is also a minimizer of $A_{\omega}$ in $\Pi_{\omega}$. In order to prove that $\phi_{1,\omega}$ and $u_1$ are radially symmetric and decreasing along the radial direction, it is sufficient to show that $\phi_{1,\omega}=\phi_{1,\omega}^\star$ a.e. on $\Rd$, where $\phi_{1,\omega}^\star$ denotes the symmetric decreasing rearrangement of $\phi_{1,\omega}$. 

For the sake of completeness, let us recall that, given a nonnegative function $f\in L^p(\R^d)$, the symmetric decreasing rearrangement $f^*$ is defined as
   \begin{equation*}
f^\star(x):=\int_0^{+\infty}\mathds{1}_{B_{r(t)}}(x)\,dt,\quad r(t):=\sqrt{\frac{|\{x\,:\,f(x)>t\}|}{\pi}}.
   \end{equation*}
   Then the following relations hold for $\phi_{1,\om}^\star$ (see for example \cite[Section 2.3]{ABCT}):
   \begin{equation}
   \label{eq:prop-rearr}
   \|\na \phi_{1,\om}^\star\|_2^2\leq \|\na \phi_{1,\om}\|_2^2,\quad \quad\|\phi_{1,\om}^\star\|_2^2=\|\phi_{1,\om}\|_2^2,
   \end{equation}
   \begin{equation}
   \label{eq:prop-rearr2}
   \forall\, p>1\quad \forall\,g \in L^p(\Rd)\quad\|\phi_{1,\om}^\star+g^\star\|_p^p\geq \|\phi_{1,\om}+g\|_p^p. 
\end{equation}
Moreover, if $g$ is radially symmetric and decreasing, then the equality in \eqref{eq:prop-rearr2} holds if and only if $\phi_{1,\om}^\star=\phi_{1,\om}$ a.e. on $\Rd$ (see again \cite[Section 2.3]{ABCT}).

   Assume now by contradiction that $\phi_{1,\omega}\neq\phi^\star_{1,\omega}$ on a set of non-zero Lebesgue measure, and define the function $U^\star:=(\phi_{1,\omega}^\star+q_1\G_\omega,u_2)$. By \eqref{eq:prop-rearr} and the fact that the equality case in \eqref{eq:prop-rearr2} is realized if and only if $\phi_{1,\om}^\star=\phi_{1,\om}$ a.e. on $\Rd$, we get that 
					$A_{\omega}(U^\star)\leq A_{\omega}(U)$
			and 
			$\|\phi_{1,\om}^\star+q_1\G_\om\|_{p_1}^{p_1}>\|\phi_{1,\om}+q_1\G_\om\|_{p_1}^{p_1}$. In particular, there exists $\tau\in(0,1)$ such that, denoted with $U^\star_\tau=(\tau\phi_{1,\om}^\star+q_1\G_\om, u_2)$,  there results that $\|\tau\phi_{1,\om}^\star+q_1\G_\om\|_{p_1}^{p_1}=\|\phi_{1,\om}+q_1\G_\om\|_{p_1}^{p_1}$, i.e. $U^\star_\tau\in \Pi_\omega$. Moreover, it holds that 
			 \begin{align}
				\frac{p_1-2}{2p_1}Q_{\omega}(U^\star_\tau) - 	\frac{p_1-2}{2p_1}Q_{\omega}(U^\star) 
				&=-\frac{p_1-2}{2p_1} (1-\tau^2)\left( \norma{\nabla\phi^\star_{1,\omega}}_2^2 + \, \omega \norma{\phi^\star_{1,\omega}}_2^2  \right)
				<0,
			\end{align}
			thus $A_\om(U^\star_\tau)<A_\om(U^\star)\leq A_\om(U)$,
			which contradicts the fact that $U$ is a minimizer of the action at frequency $\omega$, hence $\phi_{1,\omega}$ and $u_1$ are radially symmetric and decreasing. The same properties can be proved for $\phi_{2,\omega}$ and $u_2$ (using  the functional $B_\om$ instead of $A_\om$), hence the proof is concluded.\\
		
        \end{proof}
  
		\section{Proof of Theorems \ref{thm:p1=p2} and \ref{thm:sigma1=sigma2}}

\label{sec:prfthm1516}

 This section is devoted to the proof of Theorem \ref{thm:p1=p2} and Theorem \ref{thm:sigma1=sigma2}. Let us start with the proof of Theorem \ref{thm:p1=p2}.
		\begin{proof}[Proof of the Theorem \ref{thm:p1=p2}] 
  Let us start with the proof of point (i). Let $U=(u_1,u_2)$ be a ground state of \eqref{eq:F} at mass $\mu$, $q_i$ and $\mu_i$ be the charge and the mass associated with $u_i$ for $i=1,2$. By Theorem \ref{thm:charact}, we can assume that $q_i$ and $u_i$ are positive for $i=1,2$. Suppose now by contradiction that $q_1\leq q_2$ and consider the function $\widetilde{U}=(u_2,u_1)\in \La_\mu$, which satisfies 
			\begin{align}
   \label{eq:q-contr}
		F(\widetilde{U})&=F_1(u_2) + F_2(u_1) -\beta\, q_1\,q_2 \\
   &=F_1(u_1)+\frac{\sigma_1}{2}\left(q_2^2-q_1^2\right)+F_2(u_2)+\frac{\sigma_2}{2}\left(q_1^2-q_2^2\right)-\beta q_1 q_2\\
			&=F_1(u_1) + F_2(u_2)-\beta q_1\,q_2 + \f{(\sigma_{2}-\sigma_{1})}{2} (q_1^2-q_2^2)\\
   &=F(U)+\f{(\sigma_{2}-\sigma_{1})}{2} (q_1^2-q_2^2).
			\end{align}
   In particular, if $q_1<q_2$, then by \eqref{eq:q-contr} we get that $F(\widetilde{U})<F(U)$, which contradicts the fact that $U$ is a ground state. If instead $q_1=q_2=:q$, then $F(U)=F(\widetilde{U})$, hence $\widetilde{U}$ is a ground state too and solves \eqref{eq:stat-detailed}. In particular, since both $U$ and $\widetilde{U}$ solve \eqref{eq:stat-detailed}, there results that
   \begin{equation*}
   \begin{cases}
       \phi_{1,\omega}(0)=(\sigma_1+\theta_\omega-\beta)q\\
       \phi_{2,\omega}(0)=(\sigma_2+\theta_\omega-\beta)q,
   \end{cases}
   \quad\text{and}\quad
   \begin{cases}
      \phi_{2,\omega}(0)=(\sigma_1+\theta_\omega-\beta)q\\
       \phi_{1,\omega}(0)=(\sigma_2+\theta_\omega-\beta)q, 
   \end{cases}
   \end{equation*}
   hence $\sigma_1=\sigma_2$, which is again a contradiction and concludes the proof of point (1).\\
			
   Let us focus now on the proof of point (ii). Fix $\beta>0$, $\sigma_1\in \R$, denote with $p:=p_1=p_2$ and let $\sigma_2$ vary in the interval $(\sigma_1,+\infty)$: let us denote with $\delta:=\sigma_2-\sigma_1\in (0,+\infty)$. We observe that, by Theorem \ref{thm:existence},  there exists at least a ground state of \eqref{eq:F} at mass $\mu$, that we denote by $U_\delta=(u_{1,\delta}, u_{2,\delta})$. We consider for $i=1,2$ the decomposition $u_{i,\delta}=\phi_{i,\delta}+q_{i,\delta}\G_\la$, corresponding to $\la=4e^{4\pi(2\beta-\sigma_1)-2\gamma}$ or, equivalently, to $\sigma_1+\theta_\la=2\beta$, and we call $m_\delta$ the mass of $u_{1,\delta}$. By arguing as in \eqref{bound-Fun-partial} and using that $F(U_{\delta})<0$, we get that 
   \begin{equation*}
       \sum_{i=1,2}\left\{\frac{1}{4}\NN_{\sigma_i}(u_{i,\delta})^2-\frac{C_{p}\mu}{p}\NN_{\sigma_i}(u_{i,\delta})^{p-2}\right\}<\frac{\la}{2}\mu,
   \end{equation*}
   thus $\NN_{\sigma_i}(u_{i,\delta})$ are equibounded in $\delta>0$ for $i=1,2$. In particular, this entails that $(\phi_{i,\delta})_{\delta>0}$ is equibounded in $H^1(\Rd)$ and both $(u_{i,\delta})_{\delta>0}$ and $(\phi_{i,\delta})_{\delta>0}$ are equibounded in $L^r(\Rd)$ for every $r\geq 2$ for $i=1,2$. Moreover, $(q_{1,\delta})_{\delta>0}$ is equibounded and \begin{equation*}
(\sigma_2+\theta_\la)q_{2,\delta}^2=(2\beta+\delta)q_{2,\delta}^2\leq C,
   \end{equation*}
   from which we deduce that $q_{2,\delta}\to 0$ as $\delta\to+\infty$.  Since by Lagrange theorem it holds  that for every $r\geq 2$
   \begin{equation}
\left|\|u_{2,\delta}\|_{r}^{r}-\|\phi_{2,\delta}\|_{r}^{r}\right|\leq r\sup\left\{\|u_{2,\delta}\|_{r}^{r-1}, \|\phi_{2,\delta}\|_{r}^{r-1}\right\}|q_{2,\delta}|\|\G_{\la}\|_{r}\to 0,\quad \delta\to+\infty,
			\end{equation}
   it follows that
			\begin{align}
				0>F_{2}(u_{2,\delta})-E_{p}(\phi_{2,\delta})=&\f{\la}{2}\left(\norma{\phi_{2,\delta}}_2^2 - \norma{u_{2,\delta}}_2^2\right) + \f{q_{2,\delta}^2}{2}(\sigma_1+ \delta + \theta_{\la})\\
    &-\frac{1}{p}\left(\norma{{u_{2,\delta}}}_{p}^{p} - \|\phi_{2,\delta}\|_{p}^{p}\right) = \frac{\delta}{2} q_{2,\delta}^{2}+o(1),
			\end{align}
			hence $\delta q_{2,\delta}^{2}\to 0$ as $\delta\to+\infty$. Therefore,
			\begin{align}
   \label{eq:Fmu>F1Ep}
			\F(\mu)=F(U_{\delta})=F_{1}(u_{1,\delta}) + F_{2}(u_{2,\delta}) + o(1)
			&=F_{1}(u_{1,\delta}) + E_{p}(\phi_{2,\delta}) + o(1)\\
   &\geq \F_{1}(m_{\delta}) + \mathcal{E}_{p}(\mu-m_{\delta}) + o(1),\quad\delta\to+\infty.
			\end{align}
On the other hand, since $\F_{\sigma,p}$ is an increasing function of $\sigma$ by \cite[Lemma 4.3]{ABCT-24} and  \eqref{eq:Fmu=F12mu}-\eqref{eq:Fmu<F12mu} hold, we get that
				\begin{align}
				\F(\mu)=&F(U_{\delta}) \leq \min \{\F_{1}(\mu),\F_{2}(\mu)\}= \F_1(\mu),  
				\end{align}
    thus combining with \eqref{eq:Fmu>F1Ep}, there results that
    \begin{equation}
    \label{eq:contr-thm3}
        \F_1(\mu)\geq \F_1(m_\delta)+\Eps_p(\mu-m_\delta)+o(1), \quad \delta\to+\infty.
    \end{equation}

    We claim that $m_\delta\to\mu$ and $\F(\mu)\to \F_1(\mu)$. Indeed, 
if we suppose by contradiction that there exists $\delta_n\to+\infty$ as $n\to+\infty$ and  $\varepsilon>0$ such that $m_{\delta_n}\leq \mu-\varepsilon$ for every $n\in\N$, then by concavity of both $\F_1$ and $\Eps_p$ and since $\F_1(\mu)<\Eps_p(\mu)$ for every $\mu>0$, we have that 
			\begin{align}
		\lim_{n\to+\infty}[\F_{1}(m_{\delta_n}) + \mathcal{E}_p(\mu-m_{\delta_n}) +o(1)] &\geq \inf_{m\in[0,\mu-\varepsilon]} \{\F_{1}(m) + \mathcal{E}_p(\mu-m) \} \\
		&= \min \left\{ \mathcal{E}_p(\mu), \F_{1}(\mu-\varepsilon) + \mathcal{E}_p(\varepsilon)\right\} \\
		&>  \min \left\{ \mathcal{E}_p(\mu), \F_{1}(\mu)\right\} =\F_{{1}}(\mu), 
		\end{align}
				which is in contradiction with \eqref{eq:contr-thm3} and entails that $m_\delta\to \mu$ as $\delta\to+\infty$. Moreover, by \eqref{eq:Fmu>F1Ep} and \eqref{eq:contr-thm3}, we deduce that $\F(\mu)\to \F_1(\mu)$, concluding the proof of point (ii). 
				
				\end{proof}

   We are now ready to prove Theorem \ref{thm:sigma1=sigma2}.
   
   \begin{proof}[Proof of Theorem \ref{thm:sigma1=sigma2}]
   Given $\sigma=\sigma_1=\sigma_2$, we know by Theorem \ref{thm:existence} that there exists at least a ground state of \eqref{eq:F} at mass $\mu$, that we denote by $U_\sigma=(u_{1,\sigma}, u_{2,\sigma})$. Fixed $\la=4e^{-2\gamma}$, we represent $u_{i,\sigma}$ as $u_{i,\sigma}=\phi_{i,\sigma}+q_{i,\sigma}\G_\la$ for $i=1,2$. By using \eqref{eq:GN-noja} and the fact that $F(U_\sigma)<0$, we get that for $\sigma$ sufficiently large \eqref{bound-Fun} holds, hence $\NN_{\sigma}(u_{i,\sigma})$ is equibounded in $\sigma>0$. This entails that
   $\sigma q_{i,\sigma}^2$ is bounded, hence $q_{i,\sigma}\to 0$ as $\sigma \to+\infty$ for $i=1,2$.

   Therefore, as $\sigma\to+\infty$
\begin{align}
\label{eq:FmuleqEpi}
\F(\mu)=F(U_\sigma)&=\sum_{i=1,2}\left\{\frac{1}{2}\|\na \phi_{i,\sigma}\|_2^2+ \frac{1}{2}\la\left(\|\phi_{i,\sigma}\|_2^2-\|u_{i,\sigma}\|_2^2\right)-\frac{1}{p_i}\|u_{i,\sigma}\|_{p_i}^{p_i}\right\}+o(1)\\
&=\sum_{i=1,2} E_{p_i}(\phi_{i,\sigma})+ o(1)\geq \sum_{i=1,2}\mathcal{E}_{p_i}(\|\phi_{i,\sigma}\|_2^2) +o(1)\\
&=\sum_{i=1,2}\mathcal{E}_{p_i}(\|u_{i,\sigma}\|_2^2) +o(1)\geq \inf_{m\in[0,\mu]} \{\mathcal{E}_{p_1}(m) + \mathcal{E}_{p_2}(\mu-m) \}+o(1)\\
&= \min \{ \mathcal{E}_{p_1}(\mu), \mathcal{E}_{p_2}(\mu) \}+o(1).
\end{align}
On the other hand, by Lemma \ref{lem:F-inf} and Step 2 in the proof of Theorem \ref{thm:existence}, we have that $F(U_\sigma)=\F(\mu)<\F^\infty(\mu)=\min\{\Eps_{p_1}(\mu), \Eps_{p_2}(\mu)\}$, thus together with \eqref{eq:FmuleqEpi} gives
\begin{equation}
    \lim_{\sigma\to+\infty}F(U_\sigma)=\min\{\Eps_{p_1}(\mu), \Eps_{p_2}(\mu)\}.
\end{equation}

Let us distinguish now the cases $\mu<\mu^\star$ and $\mu>\mu^\star$. If $\mu<\mu^\star$, then by Proposition \ref{livellidienergia} there results that $\min\{\Eps_{p_1}(\mu), \Eps_{p_2}(\mu)\}=\Eps_{p_1}(\mu)$ and
\begin{equation}
  \label{eq:FsigmatoEp1}  
\lim_{\sigma\to+\infty}F(U_\sigma)=\Eps_{p_1}(\mu).
\end{equation}
\eqref{eq:FsigmatoEp1} entails also that $\|u_{1,\sigma}\|_2^2\to \mu$ as $\sigma\to+\infty$. Indeed, if we suppose by contradiction that there exists $\varepsilon>0$ such that $\|u_{1,\sigma}\|_2^2\leq \mu-\varepsilon$ along a subsequence of $\sigma\to+\infty$, then arguing as in \eqref{eq:FmuleqEpi} we get that
\begin{equation}
   \lim_{\sigma\to+\infty} F(U_\sigma)\geq \lim_{\sigma\to+\infty}{\sum_{i=1,2}\Eps_{p_i}(\|u_{i,\sigma}\|_2^2)}\geq\min\{\Eps_{p_1}(\mu-\varepsilon),\Eps_{p_2}(\mu)\}>\Eps_{p_1}(\mu),
\end{equation}
which contradicts \eqref{eq:FsigmatoEp1}.

If instead $\mu>\mu^\star$, then $\min\{\Eps_{p_1}(\mu), \Eps_{p_2}(\mu)\}=\Eps_{p_2}(\mu)$, hence it is possible to deduce that $F(U_\sigma)\to \Eps_{p_2}(\mu)$ and $\|u_{2,\sigma}\|_2^2\to \mu$ as done for the case $\mu<\mu^\star$, concluding the proof.
   \end{proof}

\section*{Acknowledgements}
\noindent F.B. and I. D. have been partially supported by the INdAM GNAMPA project 2023 "Modelli nonlineari in presenza di interazioni puntuali" (CUP E53C22001930001) and I. D. also by the INdAM GNAMPA project 2024 "Modelli PDE-ODE nonlineari e proprieta' di PDE su domini standard e non-standard" (CUP E53C23001670001).

\noindent R.C. acknowledges that this study was carried out within the project E53D23005450006 "Nonlinear dispersive equations in presence of singularities" - funded by European Union - Next Generation EU within the PRIN 2022 program (D.D. 104 - 02/02/2022 Ministero dell’Università e della Ricerca). This manuscript reflects only the authors' views and opinions and the Ministry cannot be considered responsible for them.

\appendix

\section{The self-adjoint operator}

\label{app:operator}

\noindent
In this paragraph, we briefly outline the rigorous construction of the linear operator. The first version that appeared in the literature is in \cite{ES-86}. 

We rigorously define the operator $H: D(H)\subset L^2(\mathcal{I})\to L^2(\mathcal{I})$. It is defined as a self-adjoint extension of the symmetric operator $H_{0} \oplus H_{0},$ where 
\[
H_0: C^{\infty}_c(\R^2\backslash\{0\})\subseteq L^2(\R^2) \to L^2(\R^2), \qquad \, H_0u=-\Delta u.
\]

Let us first recall that the self-adjoint extensions of $H_0$ are given by a one parameter family $H_{\sigma}$, with $\sigma\in\R\cup\{\infty\}$ with domain
\[
D(H_{\sigma}):=\{u\in L^2{(\R^2)}: \exists q\in\C,\lambda>0 \text{ s.t. } u-q\mathcal{G}_{\lambda}\in H^2{(\R^2)}\text{ and } \phi_{\lambda}(0)=(\sigma + \theta_{\lambda})q\}
\]
and action
\[
H_{\sigma}\,u=-\Delta\phi_{\lambda} - q \,\lambda\,\mathcal{G}_{\lambda}, \qquad \forall u=\phi_{\lambda} + q \mathcal{G}_{\lambda}\in D(H),
\]
where $\mathcal{G}_{\lambda}:\R^2\setminus\{0\}\to\Rp$ is the Green's function of $-\Delta + \lambda$ and 
\[
\theta_{\lambda} := \frac{\log(\sqrt{\lambda})-\log(2) + \gamma}{2\pi},
\]
with $\gamma$ denoting the Euler-Mascheroni costant.

We are now ready to define for every $\sigma_1,\sigma_2 \in\R$ and  $\beta\geq 0$ the operator $H:L^2(\mathcal{I}) \to L^2(\mathcal{I})$, with domain
\begin{multline}
	\label{eq:op-dom}
	D(H):= \biggl\{ U=(u_1,u_2)\in L^2(\mathcal{I}) \,:\, \exists\, q_i\in \C, \lambda_i>0\,:\,u_i-q_i\,\mathcal{G}_{\lambda_i}:=\phi_{i,\lambda_i}\in H^2(\R^2),\\
		 \begin{pmatrix}
			\phi_{1,\lambda_{1}}(0)-\theta_{\lambda_1} q_1\\
			\phi_{2,\lambda_{2}}(0)-\theta_{\lambda_2} q_2
		\end{pmatrix}= \begin{pmatrix}
		\sigma_{1} &-\beta\\
		-\beta & \sigma_{2}
	\end{pmatrix}
\begin{pmatrix}
			q_1\\
			q_2
		\end{pmatrix}
	\biggl\}
\end{multline}
and action 
\begin{equation}
	\label{eq:act-op}
	H U = (-\lap \phi_{1,\la_1}-\la_1 q_1\G_{\la_1},-\lap \phi_{2,\la_2}-\la_2 q_2\G_{\la_2}).
\end{equation}

Equation \eqref{eq:stat} can thus be rewritten as 
\begin{equation}
	\label{eq:stat-detailed}
	\begin{cases}
		(-\lap+\om)\phi_{i,\lambda_{i}} + (\omega-\lambda_i) q_i\, \mathcal{G}_{\lambda_i} - \abs{u_i}^{p_i-2}u_i=0,\quad i=1,2,\\
		\phi_{1,\la_1}(0)=(\sigma_1+\theta_{\la_1})q_1-\beta q_2,\\
  \phi_{2,\la_2}(0)=-\beta q_1+(\sigma_2+\theta_{\la_2})q_2.
	\end{cases}
\end{equation}

\begin{remark}
		It is possible to define the operator $H$ for every $\beta\in\C$: it is sufficient to substitute the condition in \eqref{eq:op-dom} with the more general condition
  \begin{equation}
  \label{eq:betacomplex}
      \begin{pmatrix}
			\phi_{1,\lambda_{1}}(0)-\theta_{\lambda_1} q_1\\
			\phi_{2,\lambda_{2}}(0)-\theta_{\lambda_2} q_2
		\end{pmatrix}= \begin{pmatrix}
		\sigma_{1} &-\beta\\
		-\overline{\beta} & \sigma_{2}
	\end{pmatrix}
\begin{pmatrix}
			q_1\\
			q_2
		\end{pmatrix}
  \end{equation}
  Nevertheless, from the variational perspective, without loss of generality we can reduce to the study of the case $\beta\geq 0$, as observed in \cite{ABCT-24}.
		\end{remark}
		
		\begin{remark}
		Let us highlight that boundary condition \eqref{eq:op-dom} does not exhaust all the possible boundary conditions that make the operator $H$ self-adjoint. All the possible boundary conditions look like
		\begin{equation*}
			(U+I)
			  \begin{pmatrix}
				L_1(u_1)\\
				L_1(u_2)
			\end{pmatrix}
			=i(U-I)\begin{pmatrix}
			L_0(u_1)\\
			L_0(u_2)
			\end{pmatrix},
			\end{equation*}
			where $L_0(\phi_\la+q\G_\la):=-\frac{q}{2\pi}$, $L_1(\phi_\la+q\G_\la):=\phi_\la(0)-\theta_\la q$ and $U$ is a 2x2 unitary matrix.
			
			 In \cite{ES-86}, the authors classified all the boundary conditions into five classes. The first one is the largest class of boundary conditions and correspond to the case discussed in the present paper (see \eqref{eq:op-dom}). Among the other classes, the only other class that permits a coupling between the two planes is the second class (see \cite[Proposition 3]{ES-86}), while the third, fourth and fith class make the two planes completely decoupled.
		\end{remark}


    

	\end{document}